\documentclass[11pt]{article}
\usepackage{amsmath,amssymb,amsthm,amsfonts,mathtools}
\usepackage{xcolor}
\usepackage{geometry}
\usepackage{microtype}
\usepackage{enumitem}
\usepackage[hidelinks]{hyperref}
\usepackage{cite}

\newtheorem{theorem}{Theorem}[section]
\newtheorem{lemma}[theorem]{Lemma}

\newtheorem{observation}{Observation}
\newtheorem{conjecture}{Conjecture}[section]
\theoremstyle{definition}
\newtheorem{definition}{Definition}[section]

\newcommand{\rev}[1]{#1}
\newcommand{\chg}[1]{#1}
\newcommand{\newchg}[1]{#1}
\newcommand{\ex}{\mathrm{ex}}
\newcommand{\cK}{\mathcal K}
\newcommand{\cF}{\mathcal F}

\newcommand{\eps}{\varepsilon}

\newcommand{\Bad}{\operatorname{Bad}}
\newcommand{\Miss}{\operatorname{Miss}}
\newcommand{\floor}[1]{\left\lfloor #1\right\rfloor}
\newcommand{\ceil}[1]{\left\lceil #1\right\rceil}
\usepackage{authblk}

\title{Tur\'an problems with bounded matching number in $k$-uniform hypergraphs}
\author{Jialin Liu, Mingyang Guo\footnote{Corresponding email: mingyangg@zzu.edu.cn}, Xiumei Wang}
\affil{School of Mathematics and Statistics, Zhengzhou University, Zhengzhou 450001, China}

\date{}

\begin{document}
	\maketitle
	\vspace{-2.5em}
	
	\begin{abstract}
		For a family $\mathcal{F}$ of $k$-graphs, $\ex_k(n,\mathcal{F})$ denotes the maximum number of edges in an $n$-vertex $\mathcal{F}$-free $k$-graph. Let $M_{s+1}^k$ denote a matching of size $s+1$ in $k$-uniform hypergraphs. Recently, Alon and Frankl (JCTB, 2024) determined $\ex_2(n,\{M_{s+1}^2,K_{\ell+1}\})$ for all $n\geq 2s+1$ and $\ell\geq 2$. For every non-bipartite graph $F$, Gerbner (JGT, 2024) determined $\ex_2(n,\{M_{s+1}^2,F\})$ for sufficiently large $n$. In this paper, we investigate this problem for different ranges of the matching parameter. First we prove that for every graph $F$ with $\chi(F)>3$, there exist constants $\beta>0$ and $s_0$ such that $\ex_2(n,\{M^2_{s+1}, F\})=\ex_2(2s+1,F)$ for $\max\{s_0,n/2-\beta n\}<s<n/2$. For integers $\ell\ge k\ge3$, let $\mathcal{K}_{\ell+1}^k$ be the family of all $k$-graphs $F$ with at most $\binom{\ell+1}{2}$ edges for which there is an $(\ell+1)$-set $L$ such that every pair of vertices of $L$ is covered by an edge of $F$, and let $H_{\ell+1}^k$ be the $k$-uniform hypergraph obtained from the complete graph $K_{\ell+1}$ by enlarging each edge with a set of $k-2$ new vertices, which is a member of $\mathcal{K}_{\ell+1}^k$. We determine $\ex_k\bigl(n,\mathcal{K}_{\ell+1}^k\cup\{M_{s+1}^k\}\bigr)$ for $s< \frac{n}{8(k-1)^{k-2}(\ell-1)^3}$. For sufficiently large $s$, we also determine 	 $\ex_k\bigl(n,\{M_{s+1}^k,H_{\ell+1}^k\}\bigr)$ for  $\frac{n}{k}-\beta n<s\leq \frac{n-k}{k}$ and $s< \frac{n}{8(k-1)^{k-2}(\ell-1)^3}$.
		
	\end{abstract}
	
	\section{Introduction}
	
	A \emph{hypergraph} $H$ is a pair $H=(V,E)$, where $V:=V(H)$ is a set of vertices and $E:=E(H)$ is a set of non-empty subsets of $V$. Let $k\geq 2$ be an integer. A hypergraph is \emph{$k$-uniform} if $E\subseteq\binom{V}{k}$, where $\binom{V}{k}:=\{T\subseteq{V}: |T|=k\}$. A $k$-uniform hypergraph is also called a \emph{$k$-graph}. 
	Note that a loopless graph is a $2$-graph.
	For $T\subseteq V$ with $1\leq |T|\leq k$, let
	$N_H(T)=\{f\in\tbinom{V}{k-|T|}:f\cup T\in E(H)\}$ and $d_H(T)=|N_H(T)|$.
	For $T=\{v\}$, write $d_H(v)=d_H(\{v\})$  and $N_H(v)=N_H(\{v\})$. For a family $\cF$ of $k$-graphs,
	 denote by $\ex_k(n,\cF)$ the maximum number of edges in an $n$-vertex $k$-graph
	that contains no member of $\cF$ as a subgraph.
	When $k=2$,   $\ex(n,\cF)$ coincides with $\ex_2(n,\cF)$.
	If $\cF=\{F\}$, we simply write $\ex_k(n,F)$ and $\ex(n,F)$.

	Let $\ell\ge k\ge2$, a partition $\{V_1,\ldots,V_\ell\}$ of $[n]$ is called \emph{balanced} if each part has size either $\lfloor n/\ell\rfloor$ or $\lceil n/\ell\rceil$. Let $\{V_1,\ldots,V_\ell\}$ be a balanced partition of $[n]$. The generalized Tur\'an graph $T_k(n,\ell)$ is the $k$-uniform hypergraph on $[n]$ whose edges are precisely the $k$-subsets containing at most one vertex from each $V_i$. We write $t_k(n,\ell)=|E(T_k(n,\ell))|$.
	In particular, $T(n,\ell):=T_2(n,\ell)$ is the complete $\ell$-partite graph on $n$ vertices with the maximum possible number of edges, and we write $t(n,\ell)=|E(T(n,\ell))|$.
	Determining Tur\'an numbers of graphs and hypergraphs is one of the central
	problems in extremal graph theory.
	Let $K_{\ell+1}$ be the complete graph on $\ell+1$ vertices.
	Tur\'an's Theorem~\cite{Turan1941} states
	that $\ex(n,K_{\ell+1})=t(n,\ell)$.
	A matching in a hypergraph is a set of pairwise disjoint edges. Let $M_{s+1}^k$ denote a matching of size $s+1$ in a $k$-graph, and $M_{s+1}$  coincides with  $M_{s+1}^2$.
	Erd\H{o}s and Gallai~\cite{ErdosGallai1959} showed that
	\[
	\ex(n,M_{s+1})=\max\left\{s(n-s)+\binom{s}{2},\binom{2s+1}{2}\right\}.
	\]

	Let $G(n,\ell,s)$  be the complete
	$\ell$-partite graph on $n$ vertices in which one part has size $n-s$ and the
	remaining $\ell-1$ parts induce $T(s,\ell-1)$. Let
	$g(n,\ell,s)=|E(G(n,\ell,s))|$. In 2024, Fu et al. \cite{FuWangYang2024} proved that $\ex(n,\{M_{s+1},K_{\ell+1}\})
	=g(n,\ell,s)$ for  $n\geq 3s+1$ and $\ell\geq 2$; Alon and Frankl~\cite{AlonFrankl2024} determined $\ex(n,\{M_{s+1},K_{\ell+1}\})$ for all $s$ and $\ell$.
	
	\begin{theorem}[Alon and Frankl~\cite{AlonFrankl2024}]\label{thm:af}
		\rev{For $n\ge2s+1$ and $\ell\ge2$,
			\[
			\ex(n,\{M_{s+1},K_{\ell+1}\})
			=\max\{t(2s+1,\ell),g(n,\ell,s)\}.
			\]}
	\end{theorem}
	
	\noindent Moreover, Alon and Frankl~\cite{AlonFrankl2024} extended their result to color-critical graphs. They showed that if $F$ is a  color-critical graph with $\chi(F)=\ell+1>2$, then, for $s>s_0(F)$ and $n>n_0(s)$,
	$\ex(n,\{M_{s+1},F\})=g(n,\ell,s).$
	Gerbner~\cite{Gerbner2024} generalized this result as follows.
	
	\begin{theorem}[Gerbner~\cite{Gerbner2024}]\label{thm:ger}
		If $\chi(F)>2$ and $n$ is large enough, then
		$\ex(n,\{M_{s+1},F\})=\ex(s,\cF)+s(n-s)$, where $\cF$ is the family of
		graphs obtained by deleting an independent set from $F$.
	\end{theorem}
	Tur\'an problems with matching constraints have attracted considerable attention in recent years~\cite{LuoZhaoLu2025, XueKang2024, ZhaoLu2024, ZhuChen2025}.
	Theorem \ref{thm:ger} determined $\ex(n,\{M_{s+1},F\})$ for fixed $s$ and sufficiently large $n$. Motivated by Theorem \ref{thm:af}, it is natural to ask for the exact value of $\ex(n,\{M_{s+1},F\})$ for all $s$ with $1\leq s< n/2$.
	We establish the following result, which determines $\ex(n,\{M_{s+1},F\})$ for $n/2-\beta n<s<n/2$. In our setting, $s$ is linear in $n$ rather than a fixed constant as in Theorem \ref{thm:ger}.
	
	\begin{theorem}\label{thm:intro-graph}
		Let $F$ be a graph with $\chi(F)>3$. There exist a positive constant $\beta$ and a sufficiently large integer $s_0$ such that, for all integers $n,s$ with 
		$\max\{s_0,n/2-\beta n\}<s<n/2$,
		$\ex(n,\{F,M_{s+1}\})=\ex(2s+1,F)$.
	\end{theorem}

	Tur\'an problems for hypergraphs are considerably more difficult.
	In 1965, Erd\H{o}s~\cite{Erdos1965} proposed the following 	conjecture.
	\begin{conjecture}[Erd\H{o}s Matching Conjecture~\cite{Erdos1965}]\label{conj1}
		Let $n,s,k$ be positive integers such that $k\ge2$ and
		$n\ge k(s+1)-1$. Then
		\[
		\ex_k(n,M_{s+1}^k)=
		\max\left\{\binom nk-\binom{n-s}{k},\binom{k(s+1)-1}{k}\right\}.
		\]
	\end{conjecture}
	\noindent This conjecture remains open for general $k$. For progress on  Conjecture \ref{conj1},  see
	\cite{AlonEtAl2012,BollobasDaykinErdos1976,ErdosKoRado1961,
		FranklRodlRucinski2012,Frankl2013,FranklKupavskii2019,
		FranklKupavskii2022,HuangLohSudakov2012,KolupaevKupavskii2023,LuYuYuan2021,
		LuczakMieczkowska2014,Frankl2017newrange,CLZ,FranklLuMaWu2026,HouHuLiu2026}.
	Many results have been obtained recently concerning the Tur\'an-type problem where the  matching $M_{s+1}^k$ with small $s$ and some hypergraph are forbidden.
	Gerbner et al.~\cite{GerbnerTompkinsZhou2025} consider forbidding hypergraphs with chromatic number $2$ and at least $3$, respectively, as well as expansions of bipartite and non-bipartite graphs for fixed $s$.
	Wang et al.~\cite{WangWangYang2025} studied  $F_5$-free $3$-graphs for $n\geq 30(s+1)$, where $F_5$ has vertex set $\{a,b,c,d,e\}$ and edge set $\{\{a,b,c\}, \{a,b,d\}, \{c,d,e\}\}$.
	Chen et al.~\cite{ChenLiuQiYang2025} studied $F_{3,2}$-free $3$-graphs for $n\geq 12s^2$, where $F_{3,2}$ has vertex set $\{a,b,c,d,e\}$ and edge set $\{\{a,b,c\},\{a,d,e\},\{b,d,e\},\{c,d,e\}\}$.
	Xu et al.~\cite{XuZengZhang2026} investigated Berge-$K_3$-free $3$-graphs for $n\geq 3s$ and $4$-graphs for $n\geq 4s$.
	Zhou and Yuan~\cite{ZhouYuan2026}  explored  linear Tur\'an problems for expansions of graphs.
	
	In this paper, we study Tur\'an-type problems forbidding both   $M^k_{s+1}$ and fixed hypergraphs in the regime where $s$ is comparable to $n$.
	Let $\mathcal{K}_{\ell+1}^k$ be the family of all $k$-graphs $F$ with at most $\binom{\ell+1}{2}$ edges for which there exists an $(\ell+1)$-set $L$ such that every pair of vertices in $L$ is contained in some edge of $F$.
	Any such set $L$ is called a \emph{core} of $F$.
	Let $H_{\ell+1}^k$ be the $k$-graph obtained from the complete graph $K_{\ell+1}$ by enlarging each edge with a set of $k-2$ new vertices, where all these sets are pairwise disjoint. Then $H_{\ell+1}^k\in\mathcal{K}_{\ell+1}^k$. Mubayi~\cite{Mubayi2006} proved the following theorem.
	
	\begin{theorem}[Mubayi~\cite{Mubayi2006}]\label{thm:weak-turan}
		Let $n\ge\ell\ge k\ge2$. Then
		$\ex_k(n,\cK_{\ell+1}^k)=t_k(n,\ell)$, and $T_k(n,\ell)$ is the unique
		extremal $k$-graph.
	\end{theorem}
	Later, Pikhurko \cite{Pikhurko} determined $\ex_k(n,H_{\ell+1}^k)$ for sufficiently large $n$.
	For $k\ge 3$ and $n/k-\beta n<s\leq(n-k)/k$, we investigate $k$-uniform hypergraphs forbidding both $H_{\ell+1}^k$
	and $M^k_{s+1}$ with separate discussions for $\ell>k$ and $\ell=k$. Our results demonstrate that these two parameter regimes yield entirely different extremal constructions.
	
	\begin{theorem}\label{thm:intro-hypergraph}
		For every $\ell>k\ge3$, there exist $\beta>0$ and $s_0$ such that, for
		all integers $n,s$ with  $\max\{s_0,n/k-\beta n\}<s\leq(n-k)/k$,
		$\ex_k\bigl(n,\{H_{\ell+1}^k,M_{s+1}^k\}\bigr)
		=\ex_k\bigl(ks+k-1,H_{\ell+1}^k\bigr)$.
	\end{theorem}
	
	\begin{theorem}\label{thm:intro-boundary}
		For every $k\ge3$, there exist $\beta>0$ and
		$s_0$ such that, for all integers $n,s$ with $\max\{s_0,n/k-\beta n\}<s\leq (n-k)/k$,
		$\ex_k\bigl(n,\{H_{k+1}^k,M_{s+1}^k\}\bigr)
		=s\,t_{k-1}(n-s,k-1)$.
	\end{theorem}
	

	Given a $k$-graph $H$, the \emph{link} of $v\in V(H)$ is defined as
	$L_H(v)=\left\{A\in\binom{V(H)}{k-1}:\{v\}\cup A\in H\right\}$.
	Thus $L_H(v)$ is a $(k-1)$-graph and $d_H(v)=|L_H(v)|$.
	Let $\ell\ge k\ge3$, $s\ge1$, and $n\ge\ell+s$. Let
	$\{V_0, V_1, \cdots, V_{\ell-1}\}$ be a partition of a set $V$ of size $n$ such that $|V_0|=s$ and
	$|V_i|=\floor{\frac{n-s+i-1}{\ell-1}}$ for each $i\in[\ell-1]$. Then $\{V_1, \cdots, V_{\ell-1}\}$ is a balanced partition of $V\setminus{V_0}$.
	Define the $k$-graph $H(n,\ell,s,k)$ with vertex set $V$ and edge set
	\[
	E(H(n,\ell,s,k)):=
	\left\{e\in\binom{V_0\cup\cdots\cup V_{\ell-1}}k:
	|e\cap V_0|=1,\ |e\cap V_i|\le1\ \text{for each} \ i\in[\ell-1]\right\}.
	\]
	Note that in the hypergraph $H(n,\ell,s,k)$, for every $x\in V_0$,
	$L_{H(n,\ell,s,k)}(x)=T_{k-1}(n-s,\ell-1)$, and hence
	$|E(H(n,\ell,s,k))|=s\,t_{k-1}(n-s,\ell-1)$.
	Since every edge meets $V_0$, the matching number of $H(n,\ell,s,k)$ is at most $s$.
	Since every $(\ell+1)$-set of $V$ has
	two vertices in a common part $V_i$ and no edge of $H(n,\ell,s,k)$ contains
	both of them,  	$H(n,\ell,s,k)$ is $\cK_{\ell+1}^k$-free. In particular, $H(n,\ell,s,k)$ is $H_{\ell+1}^k$-free. Yang et al. \cite{YangZengZhang2025} determined $\ex_k(n,\cK_{\ell+1}^k\cup\{M_{s+1}^k\})$ for fixed $s$ and sufficiently large $n$.
	Zhao et al.\cite{ZhaoWangZhou2026} determined $\ex_k(n,\{H_{\ell+1}^k,M_{s+1}^k\})$ for fixed $s$ and sufficiently large $n$.
	
	\begin{theorem}[Yang, Zeng and Zhang\chg{~\cite{YangZengZhang2025}}]
		Let $\ell\ge k\ge3$ and $s\ge1$. For sufficiently large $n$,
		\[
		\ex_k(n,\cK_{\ell+1}^k\cup\{M_{s+1}^k\})
		=s\,t_{k-1}(n-s,\ell-1),
		\]
		and $H(n,\ell,s,k)$ is the unique extremal hypergraph.
	\end{theorem}
	
	\begingroup

	\begin{theorem}[Zhao, Wang and Zhou~\cite{ZhaoWangZhou2026}]
		Let $\ell\ge k\ge3$. There exists $s_0=s_0(k,\ell)$ such that, for every
		fixed integer $s\ge s_0$ and all sufficiently large $n$,
		\[
		\ex_k\bigl(n,\{H_{\ell+1}^k,M_{s+1}^k\}\bigr)
		=s\,t_{k-1}(n-s,\ell-1).
		\]
	\end{theorem}
	\endgroup
	We strengthen these results by establishing the same formula for a linear range of $s$.
	
	\begin{theorem}\label{thm:small-s}
		Let $n, s,k,\ell$ be integers such that $\ell\geq k\geq 3$ and $1\le s<\frac{n}{8(k-1)^{k-2}(\ell-1)^3}$. Then
		\[
		\ex_k\bigl(n,\cK_{\ell+1}^k\cup\{M_{s+1}^k\}\bigr)
		=s\,t_{k-1}(n-s,\ell-1).
		\]
	\end{theorem}
	
	\begin{theorem}\label{thm:small-expansion}
		Let $n,s, k,\ell$ be integers such that $\ell\geq k\geq 3$. There are constants
		$s_0=s_0(k,\ell)$ and $n_0=n_0(k,\ell)$ such that for
		$s_0\le s<\frac{n}{8(k-1)^{k-2}(\ell-1)^3}$ and $n\ge n_0$,
		\[
		\ex_k\bigl(n,\{H_{\ell+1}^k,M_{s+1}^k\}\bigr)
		=s\,t_{k-1}(n-s,\ell-1).
		\]
	\end{theorem}
	
	The remainder of the paper is organized as follows. Section~2 collects the preliminary results. In Section~3, we prove Theorems~\ref{thm:intro-graph} and~\ref{thm:intro-hypergraph}, while Section~4 is devoted to the proof of Theorem~\ref{thm:intro-boundary}. Section~5 contains the proofs of Theorems~\ref{thm:small-s} and~\ref{thm:small-expansion}. Section~6 presents some concluding remarks, and the proofs of the technical numerical lemmas are deferred to the Appendix.

	\section{Preliminaries}
	Let $H$ be a  $k$-graph and  $S\subseteq V(H)$. We call  $S$ \emph{strongly independent} if no two vertices in $S$ lie in a common edge of $H$ and  \emph{weakly independent} if $S$ contains no edge of $H$.
	We use $H-S$ to denote the hypergraph obtained from $H$ by deleting $S$ and all edges of $H$ intersecting  $S$, and use $H[S]$ to denote the sub-hypergraph with vertex set $S$ and edge set $\{e\in E(H): e\subseteq S\}$. For   $E'\subseteq E(H)$, we use $H-E'$ to denote the hypergraph obtained from $H$ by deleting $E'$. The size
	of the largest matching in $H$ is denoted by $\nu(H)$. A matching of $H$ is perfect
	if it covers all vertices of $H$.
	Given two $k$-graphs $H_1$ and $H_2$ of order $n$, we say that $H_1$ is \emph{$\varepsilon$-close} to $H_2$ if $H_1$ can be transformed to $H_2$ by adding and deleting at most $\varepsilon n^k$ edges. For a family of graphs $\cF$, define $\chi(\cF)=\min\{\chi(F):F\in \cF\}$.
	
	The classical stability theorem was proved independently by
	Erd\H{o}s and Simonovits~\cite{Erdos1968Stability,Simonovits1968}.
	
	\begin{theorem}[\chg{Erd\H{o}s-Simonovits
			\cite{Erdos1968Stability,Simonovits1968}}]\label{erdos-stability}
		Let $\ell\ge2$ and let $\cF$ be a family of graphs with
		$\chi(\cF)=\ell+1$. For every $\eps>0$, there exist $\delta>0$ and $n_0$ such
		that every $\cF$-free graph on $n\ge n_0$ vertices with at least
		$(1-1/\ell)\binom n2-\delta n^2$ edges is $\eps$-close to $T_2(n,\ell)$.
	\end{theorem}
	
	
	
	Pikhurko~\cite{Pikhurko} proved the following stability result for $H_{\ell+1}^k$.
	\begin{theorem}[Pikhurko~\cite{Pikhurko}]\label{thm:exp-stability}
		For any $\ell\geq k\geq 3$ and any $\varepsilon>0$, there are $\delta>0$ and $n_0=n_0(k,\ell,\varepsilon)$ such that
		the following holds for all $n\ge n_0$. If a $k$-graph $H$ on $n$
		vertices is $H_{\ell+1}^k$-free and has at least
		$t_k(n,\ell)-\delta n^k$ edges, then $H$ is $\varepsilon$-close to $T_k(n,\ell)$.
	\end{theorem}
	
	A $k$-graph $G$
	is \emph{$\ell$-partite} if its vertex set has a partition into
	\rev{$\ell$ classes $V_1,V_2,\ldots,V_\ell$} such that
	$|e\cap V_i|\le1$ for every edge $e$. Given such a partition, a set
	$T\subseteq V(G)$ is \emph{legal} if $|T\cap V_i|\le1$ for every $i\in [\ell]$.
	The following lemma shows that a nearly balanced $\ell$-partite $k$-graph
	with large minimum degree has a near-perfect matching. It will be used in Section 3 to prove Lemma \ref{lem:close-reduction}.
	
	\begin{lemma}\label{lem:partite-matching}
		For every $\ell>k\ge2$, there is $\eps_0=\eps_0(k,\ell)>0$ such that
		the following holds for all sufficiently large $n$. Let
		$0<\delta\le\eps\le\eps_0$, and let $G$ be an $\ell$-partite $k$-graph with
		vertex set $V_1\cup\cdots\cup V_\ell$ such that
		$(1+\delta)n\ge |V_\ell|\ge\cdots\ge|V_1|\ge n$.
		If
		$d_G(v)>(\binom{\ell-1}{k-1}-\eps) n^{k-1}$ for every $v\in V(G)$,
		then $G$ has a matching covering all but at most $k-1$ vertices.
	\end{lemma}
	
	\begin{proof}
		Let $K$ be the complete $\ell$-partite $k$-graph with parts $V_1, \cdots, V_\ell$. Then $G$ is a subgraph of $K$, and for every vertex $v$ of $K$, $d_K(v)\le\binom{\ell-1}{k-1}(1+\delta)^{k-1}n^{k-1}$. Since $(1+\delta)^{k-1}\leq 1+k\delta$ for sufficiently small $\delta$, we have
		\begin{equation}\label{eq:missing-link}
			|N_K(v)\setminus N_G(v)|=d_K(v)-d_G(v)\leq (k\delta\binom{\ell-1}{k-1}+\varepsilon)n^{k-1}
			\le 2k\binom{\ell-1}{k-1}\eps n^{k-1}.
		\end{equation}

		Suppose, to the contrary, that $\nu(G)<\lfloor\frac{|V(G)|}{k}\rfloor$. Since $K$ has a matching of size $\lfloor\frac{\ell n}{k}\rfloor$,  we may assume that $\nu(G)\geq \lfloor\frac{ (k+1)n}{k}\rfloor-1> (1+\delta)n$ by adding edges if necessary.
		Choose a maximum matching $M$ of $G$ for which the number of parts containing
		uncovered vertices is as large as possible. Since $|M|> n$, there are $k$ parts such that $|M\cap E(K[V_{j_1}\cup\cdots\cup V_{j_k}])|> \frac{n}{\binom\ell k}$
		for all sufficiently large $n$. Let $M'=M\cap E(K[V_{j_1}\cup\cdots\cup V_{j_k}])$. 
		
		\vspace{0.15cm}	
		{\bf Claim.}
		If $M^*\subseteq M'$ and
		$|M^*|\ge n/(2\binom{\ell}{k})$, then for every legal
		$k$-set $T$ disjoint from $V(M^*)$, $G[V(M^*)\cup T]$ has a matching of size $|M^*|+1$.
		\vspace{0.15cm}	
		
		We now prove this claim. Write
		$T=\{v_{a_1},\ldots,v_{a_k}\}$, where
		$v_{a_i}\in V_{a_i}$ for $1\leq i\leq k$. For
		$M''\in\binom{M^*}{k-1}$, let
		\[
		\mathcal T(M'')
		=
		\left\{
		S\subseteq V(M''): S \text{ is legal and }
		|S\cap e|=1\text{ for every }e\in M''
		\right\}.
		\]
		We say that $M''$ is a \emph{$(k-1)$-neighbour} of
		$v_{a_i}$ if
		$\{v_{a_i}\}\cup S\in E(G)$
		for every $S\in\mathcal T(M'')$ with
		$S\cap V_{a_i}=\varnothing$. Denote the family of all
		$(k-1)$-neighbours of $v_{a_i}$ by $\mathcal M_{a_i}$.
		Note that for each $M''\in\binom{M^*}{k-1}\setminus\mathcal M_{a_i}$, there exists $S\in\mathcal T(M'')$ such that $S\in N_K(v_{a_i})\setminus N_G(v_{a_i})$.
		Since $M^*$ is a matching and $M''$ consists precisely of the edges of $M^*$ intersecting $S$, we have 
		$\left|
		\binom{M^*}{k-1}\setminus\mathcal M_{a_i}
		\right|
		\le 2k\binom{\ell-1}{k-1}\eps n^{k-1}$
		by \eqref{eq:missing-link}.
		Consequently,
		$|\mathcal M_{a_i}|
		\ge
		\binom{|M^*|}{k-1}-2k\binom{\ell-1}{k-1}\eps n^{k-1}$.
		
		Choose $M''\in\binom{M^*}{k-1}$ uniformly at random. Since $|M^*|\ge \frac{n}{2\binom{\ell}{k}}$, $n$ is sufficiently large  and $\eps$  is sufficiently small, by the union bound,
		\begin{equation}
			\begin{aligned}
				\mathbb P\left(M''\in\bigcap_{i=1}^{k}\mathcal M_{a_i}\right)
				\ge 1-\sum_{i=1}^{k}\mathbb P(M''\notin\mathcal M_{a_i})
				\ge 1-\frac{2k^2\binom{\ell-1}{k-1}\eps n^{k-1}}{\binom{|M^*|}{k-1}}>0.
			\end{aligned}
		\end{equation}
		Thus there exists
		an	$M_0\in\bigcap_{i=1}^{k}\mathcal M_{a_i}$.
		Write
		$M_0=\{e_1,\ldots,e_{k-1}\}$ and
		$e_t=\{x_{t,1},\ldots,x_{t,k}\}$, where
		$x_{t,i}\in V_{j_i}$ for $i\in [k]$. Since $T$ is legal, there is a bijection $\pi:T\to[k]$
		such that $\pi(u)=i$ whenever $u\in V_{j_i}$, $i\in[k]$. For each $i\in[k]$, define
		\[
		f_i=
		\{x_{1,i+1},x_{2,i+2},\ldots,x_{k-1,i+k-1}\},
		\]
		where the second subscripts are modulo $k$.
		Then $\{f_1,\ldots,f_k\}$ is a partition of $V(M_0)$ and $f_i\cap V_{j_i}=\varnothing$ for every
		$i\in[k]$. 
		For every $u\in T$, since $M_0$ is a $(k-1)$-neighbour of $u$, we have	$\{u\}\cup f_{\pi(u)}\in E(G)$.
		Therefore, $\{\{u\}\cup f_{\pi(u)}:u\in T\}$
		is a perfect matching of size $k$ in $G[V(M_0)\cup T]$. It follows that $\{\{u\}\cup f_{\pi(u)}:u\in T\}\cup (M^*\setminus M_0)$ is a matching of size $|M^*|+1$ in $G[V(M^*)\cup T]$. This proves the Claim.

		
		Let $\mathcal V$ be the set of parts  containing  vertices uncovered by $M$. If
		$|\mathcal V|\ge k$, choose an uncovered legal $k$-set $T$. Applying the
		Claim  with $M^*=M'$, we can  find a matching of size $|M|+1$. Then we get a contradiction to the choice of $M$. Thus
			$|\mathcal V|\le k-1$.
		
		Since $|M|\leq\floor{|V(G)|/k}-1$, at least $k$ vertices are
		uncovered by $M$. By $|\mathcal V|\le k-1$, some $V_t\in\mathcal V$
		contains two uncovered vertices. Fix one of them, say $y$. Recall that
		$|M|>(1+\delta)n\geq |V_t|$.  Thus there is an edge
		$e_0\in M$ avoiding $V_t$. As every edge intersects $k$ distinct parts and
		$|\mathcal V|\le k-1$, the edge $e_0$ intersects a part $V_i\notin\mathcal V$.
		Fix $x\in e_0\cap V_i$ and set $T_0=(e_0\setminus\{x\})\cup\{y\}.$
		This is a legal $k$-set disjoint from
		$V(M'\setminus\{e_0\})$. Applying the claim with $M^*=M'\setminus\{e_0\}$ and $T=T_0$,
		we obtain a matching $M_1$ of size $|M'\setminus\{e_0\}|+1$
		in $G[V(M'\setminus\{e_0\})\cup T_0]$.
		Note that $M_1$ covers every vertex of
		$V(M'\setminus\{e_0\})\cup T_0$.
		Let $M_2=M_1\cup\bigl(M\setminus(M'\cup\{e_0\})\bigr)$.
		Then $M_2$ is a matching with $|M_2|=|M|$ and
		$V(M_2)=(V(M)\setminus\{x\})\cup\{y\}$.
		Thus the number of parts
		containing vertices uncovered by $M_2$ is greater than
		the corresponding number for $M$, contradicting the choice
		of $M$. This completes the proof.
	\end{proof}
	
	Let $H$ be an $n$-vertex $k$-graph and $u,v$ be two distinct vertices in $V(H)$. Write $d_H(u,v)=d_H(\{u,v\})$  and $N_H(u,v)=N_H(\{u,v\})$. The set $\{u,v\}$ is called \emph{sparse} in $H$ if $d_H(u,v)\le k\binom{\ell+1}{2}n^{k-3}$, and \emph{dense} in $H$ otherwise. For $v\in V(H)$, if $\{v,u\}$ is a sparse (dense) pair, then $u$ is called a \emph{sparse (dense) neighbour} of $v$. We say that $\{u,v\}$ \emph{joins} two vertex sets if $u$ belongs to one set and $v$ belongs to the other set.
	The following two lemmas will be used several times in the proof of the main results.
	
	\begin{lemma}\label{obs:extension}
		Let $n,\ell,k$ be integers such that $\ell\ge k\ge3$ and $H$ be an $n$-vertex $k$-graph.
		Let $C\subseteq V(H)$ be a subset of size $\ell+1$ such that all pairs in $\binom{C}{2}$ are dense apart from $\{a,b\}$.
		If there exists an edge $e$ satisfying $e\cap C=\{a,b\}$, then there is a copy of
		$H_{\ell+1}^k$ with core $C$ in $H$.
	\end{lemma}
	
	\begin{proof}
		Let $f_1=e$ and suppose that there are edges $f_1,\ldots,f_t$ such that  for  $i\in [t]$, $f_i\cap C\in \binom{C}{2}$, $f_i\cap C$ are distinct and $f_i\setminus C$ are pairwise disjoint. Note that $t\geq 1$. If $t=\binom{\ell+1}{2}$, then the proof is done. So we assume that $1\le t<\binom{\ell+1}{2}$. Let $\{x,y\}\in \binom{C}{2}$ be a pair such that there is no $f_i$ containing $\{x,y\}$.  Since $\{x,y\}$ is dense, $d_H(x,y)> k\binom{\ell+1}{2}n^{k-3}$. Thus there is an edge $f_{t+1}$ containing $x$ and $y$ such that $f_{t+1}\cap C=\{x,y\}$ and $f_i\setminus C$, $i\in [t+1]$,  are pairwise disjoint. Continuing this process, we can find the desired $H_{\ell+1}^k$ with core $C$ in $H$.
	\end{proof}

	\begin{lemma}\label{lem:selection}
		Let $n,k,r$ be integers such that $r\geq k\geq 3$. Suppose that $H$ is an $n$-vertex $k$-graph and $\{V_1, \cdots, V_r\}$ is a partition of $V(H)$. Let $I$ be a subset of $[r]$ of size at least $2$  and for each $i\in I$, let $Z_i$ be a subset of $V_i$.
		Let $\mathcal S$ be the family of all sparse pairs joining $Z_i$ and $Z_j$ for distinct $i,j\in I$ and let $m=\min_{i\in I}|Z_i|$.
		If $|\mathcal S|<m^2$,  then there exist vertices $z_i\in Z_i$, $i\in I$, such that every pair in $\{z_i:i\in I\}$ is dense in $H$.
	\end{lemma}
	
	\begin{proof}
		Let $\mathcal T$ be the complete $|I|$-partite $|I|$-graph with parts $Z_i$, $i\in I$.
		Then $|E(\mathcal T)|=\prod_{i\in I}|Z_i|$. A sparse pair joining
		$Z_i$ and $Z_j$ belongs to
		$\prod_{h\in I\setminus\{i,j\}}|Z_h|$
		edges of $\mathcal T$. Since $\prod_{h\in I\setminus\{i,j\}}|Z_h|
		=\frac{|E(\mathcal T)|}{|Z_i||Z_j|}
		\le\frac{|E(\mathcal T)|}{m^2}$, the number of edges of $\mathcal T$
		containing at least one sparse pair is at most
		$\frac{|E(\mathcal T)|}{m^2}|\mathcal S|$. Note that $\frac{|E(\mathcal T)|}{m^2}|\mathcal S|<|E(\mathcal T)|$.
		Therefore, for each $i\in I$,
		there exists $z_i\in Z_i$  such that every pair among
		$\{z_i: i\in I\}$ is dense in $H$.
	\end{proof}

	\section{Proof of Theorems~\ref{thm:intro-graph} and~\ref{thm:intro-hypergraph}
	}
	
	In this section, we use the following lemma to prove Theorems~\ref{thm:intro-graph} and~\ref{thm:intro-hypergraph}.
	\begin{lemma}\label{lem:close-reduction}
		Let $k,\ell$ be integers such that $\ell>k\ge2$ and $\cF$ be a finite family
		of $k$-graphs. There exist $\eps,\beta>0$ and a sufficiently large integer $s_0$ such that the following holds. Let $n,s$ be
		integers with $\max\{s_0,n/k-\beta n\}<s<n/k$. If an $n$-vertex $k$-graph
		$G$ is $\cF$-free, satisfies $\nu(G)\le s$, and is
		$\eps$-close to $T_k(n,\ell)$, then
		$|E(G)|\le\ex_k(ks+k-1,\cF)$.
	\end{lemma}
	
	\begin{proof}
		Since $G$ is $\varepsilon$-close to $T_k(n,\ell)$, there is a
		copy $T$ of $T_k(n,\ell)$ defined on $V(G)$ such that
		\[
		|E(G)\setminus E(T)|+|E(T)\setminus E(G)|
		\leq \varepsilon n^k.
		\]
		Let $X=\{x\in V(G):|N_T(x)\setminus N_G(x)|
		>\sqrt{\varepsilon}n^{k-1}\}.$
		We claim that $|X|\leq k\sqrt{\varepsilon}n$, otherwise
		\[
		|E(T)\setminus E(G)|
		\geq \frac{1}{k}\sum_{x\in X}|N_T(x)\setminus N_G(x)|
		>\varepsilon n^k,
		\]
		a contradiction.
		
		Let $U=V(G)\setminus X$ and let $H$ be the $k$-graph defined on $U$ with edge
		set $E(G[U])\cap E(T[U])$. Then $H$ is an $\ell$-partite $k$-graph.
		Denote the vertex partition of $H$ by $\{U_1,U_2,\ldots,U_\ell\}$. For each vertex
		$x\in U$, we have
		\begin{equation}\label{H-degree}
			\begin{split}
				d_H(x)
				&\geq |N_G(x)\cap N_T(x)|-|X|n^{k-2}\\
				&\geq d_T(x)-|N_T(x)\setminus N_G(x)|
				-k\sqrt{\varepsilon}n^{k-1}\\
				&\geq d_T(x)-(k+1)\sqrt{\varepsilon}n^{k-1}.
			\end{split}
		\end{equation}
		
		Let $F=\{e\in E(G):e\cap X\neq\emptyset\}$ and $M\subseteq F$ be a
		matching in $G$ which maximizes $|U\cup V(M)|$.  Note that $\{U_1,U_2,\ldots,U_\ell\}$ is the vertex partition of $T-X$. Thus $\frac{n}{\ell}-|X|-1\leq |U_i|\leq \frac{n}{\ell}+1$ for $i\in[\ell]$.
		Since every edge of $M$ intersects $X$ and $M$ is a matching, we have $|M|\leq |X|\leq k\sqrt{\varepsilon}n$. Thus
		\[
		\frac{n}{\ell}-(k+k^2)\sqrt{\varepsilon}n-1
		\leq |U_i\setminus V(M)|
		\leq \frac{n}{\ell}+1
		\]
		for each $i\in[\ell]$. Moreover, for every
		$x\in U\setminus V(M)$, inequality \eqref{H-degree} implies that
		\begin{equation}\label{H-minus-M-degree}
			\begin{split}
				d_{H[U\setminus V(M)]}(x)
				&\geq d_T(x)-(k+1)\sqrt{\varepsilon}n^{k-1}
				-|V(M)|n^{k-2}\\
				&\geq
				\binom{\ell-1}{k-1}
				\left(\frac{n}{\ell}-1\right)^{k-1}
				-(k^2+k+1)\sqrt{\varepsilon}n^{k-1}\\
				&\geq
				\left(\binom{\ell-1}{k-1}
				-2k^2\ell^{k}\sqrt{\varepsilon}\right)
				\left(\frac{n}{\ell}\right)^{k-1}.
			\end{split}
		\end{equation}
		 Therefore, by choosing $\varepsilon$ sufficiently
		small and $n$ sufficiently large, applying Lemma~\ref{lem:partite-matching}
		to $H[U\setminus V(M)]$ yields a matching $M'$ of size
		$|M'|=\left\lfloor\frac{|U\setminus V(M)|}{k}\right\rfloor.$
		
		If
		$|V(M)\cup U|\geq ks+k$, then
		\[
		\begin{split}
			|M\cup M'|
			=\frac{|V(M)|}{k}
			+\left\lfloor\frac{|U\setminus V(M)|}{k}\right\rfloor
			=\left\lfloor\frac{|V(M)\cup U|}{k}\right\rfloor
			\geq\frac{ks+k}{k}=s+1,
		\end{split}
		\]
		a contradiction. 	Thus $|V(M)\cup U|\leq ks+k-1$.
		
		Let $I=V(G)\setminus(V(M)\cup U)$ and $W=\{x\in V(M):d_G(x)\leq
		2k^2\sqrt{\varepsilon}n^{k-1}\}.$
		Let $M_1=\{e\in M:e\cap W=\emptyset\}$ and $M_2=M\setminus M_1$.
		
		\vspace{0.15cm}	
		{\bf Claim.}		
		Every edge of $G$ intersecting $I$ also intersects $V(M_2)$.
		\vspace{0.15cm}	
		
		Suppose, to the contrary,  that there is an edge $f_0\in E(G)$  such that $f_0\cap I\neq\emptyset $ and
		$f_0\cap V(M_2)=\emptyset$.
		Let $M_0=\{e\in M_1:f_0\cap e\neq\emptyset\}$. Since $M$ is a matching in $G$ which maximizes $|U\cup V(M)|$, $f_0\cap V(M_1)\neq \emptyset$. Thus $M_0\neq \emptyset$.
		Let $A=V(M_0)\setminus(f_0\cup U).$
		Since $M_0\subseteq M_1$, the definition of $M_1$ implies that $V(M_0)\cap W=\emptyset$. Thus $d_G(x)>2k^2\sqrt{\varepsilon}n^{k-1}$ for each $x\in A$.
		
		We assert that there exists a matching
		$M'_0=\{h_x:x\in A\}$ such that
		$h_x\cap A=\{x\}$ for every $x\in A$ and $V(M'_0)\setminus A
		\subseteq U\setminus\bigl(V(M)\cup f_0\bigr)$.
		Indeed, suppose that a matching $M''_0$ has already been chosen with respect to
		some vertices of $A$, and let $x\in A\setminus V(M''_0)$. Define
		$B_x=\bigl(V(M)\setminus\{x\}\bigr)
		\cup I\cup f_0\cup V(M''_0).$
		Since $|M|\leq k\sqrt{\varepsilon}n$, $|I|\leq|X|\leq
		k\sqrt{\varepsilon}n$ and $|M_0|\leq k$, we have
		$|B_x|<2k^2\sqrt{\varepsilon}n$
		for sufficiently large $n$. Hence the number of edges containing $x$
		and meeting $B_x$ is less than $2k^2\sqrt{\varepsilon}n^{k-1}$,  which is less than $d_G(x)$.
		We may therefore choose an edge
		$h_x$ containing $x$ and avoiding $B_x$. Since
		$X\subseteq V(M)\cup I,$
		all the other vertices of $h_x$ belong to
		$U\setminus(V(M)\cup f_0\cup V(M''_0))$. Continuing this process, we can find
		the desired matching $M'_0$.
		
		Let
		$\widetilde M=(M\setminus M_0)\cup\{f_0\}\cup M'_0$. Then $\widetilde M$
		is a matching contained in $F$, 
		$V(M)\setminus U\subseteq V(\widetilde M)$ and $ V(\widetilde M)\cap I\neq \emptyset$. Consequently,
		$|U\cup V(\widetilde M)|>|U\cup V(M)|,$
		a contradiction  to the choice of $M$. This completes the proof of the Claim.

	By the Claim, $f\cap V(M_2)\neq\emptyset$ for each $f\in\{e\in E(G):e\cap I\neq\emptyset\}$.
		If $M_2=\emptyset$, then there is no edge of $G$
		intersecting $I$. Hence the number of isolated vertices of $G$ is at least $n-(ks+k-1)$. Therefore
		$|E(G)|\leq\operatorname{ex}_k(ks+k-1,\cF)$
		and the proof is done. Thus, in the following proof,
		we may assume that $M_2\neq\emptyset$.	Recall that for each $e\in M_2$, $e\cap W\neq \emptyset$. Let $M_2=\{e_1,\ldots,e_{|M_2|}\}$. For each
		$i\in[|M_2|]$, choose one vertex $x_i\in e_i\cap W$, and let
		$U_0=\{x_i:1\leq i\leq |M_2|\}.$
		Then $|U_0|=|M_2|$ and $|V(M_2)|=k|U_0|$.
		
		We will construct an $\cF$-free $k$-graph $G'$  on $V(G)$ such that
		$\nu(G')\leq s$ and $|E(G)|<|E(G')|\leq\operatorname{ex}_k(ks+k-1,\cF)$.
		Let $q=\max\{|V(F')|:F'\in\cF\}$. Since $n$ is sufficiently large, we can choose $q$
		vertices from $U_1\setminus U_0$. Denote the set of these $q$ vertices
		by $S$. Let $N(S)=\bigcap_{x\in S}N_H(x)$ and
		$$\mathcal D=
		\left\{
		A\in\binom{U}{k-1}:
		A\cap U_1=\emptyset
		\text{ and }|A\cap U_i|\leq1\text{ for every }i\in[\ell]
		\right\}.$$
		Since all vertices of $S$ lie in $U_1$, we have $ N(S)\subseteq \mathcal D$. Moreover, the definition of $X$ gives
		$|\mathcal D\setminus N_H(x)|
		\leq\sqrt{\varepsilon}n^{k-1}$ for every $x\in S$.
		As $|U_i|\geq n/\ell-k\sqrt{\varepsilon}n-1$ for every $i\in[\ell]$, we have
		\begin{equation}\label{common-neighborhood}
			\begin{split}
				|N(S)|
				\geq |\mathcal D|
				-\sum_{x\in S}|\mathcal D\setminus N_H(x)|
				\geq |\mathcal D|-q\sqrt{\varepsilon}n^{k-1}
				\geq \frac{n^{k-1}}{2\ell^{k-1}},
			\end{split}
		\end{equation}
		provided that $\varepsilon$ is sufficiently small and $n$ is
		sufficiently large.
		
		We now construct $G'$ by deleting every edge of $G$
		intersecting $I\cup U_0$ and adding all sets
		$\{x\}\cup A$ with $A\in N(S)$ and $A\cap U_0=\emptyset$ for every $x\in U_0$.
		Then $I$ is an isolated vertex set in $G'$. Since
		$|U\cup V(M)|\leq ks+k-1$, it follows that $\nu(G')\leq s$.
		
		We next prove that $G'$ is $\cF$-free. Suppose that $G'$ contains a copy
		$Q$ of some $F'\in\cF$. If $V(Q)\cap U_0=\emptyset$, then every edge of
		$Q$ is an edge of $G$, and hence $G$ contains a copy of $F'$, a
		contradiction. Thus
		$C_0=V(Q)\cap U_0\neq\emptyset$.
		Every edge of $G'$ intersecting $U_0$ contains exactly one vertex of $U_0$.
		Since $|S|=q\geq |V(F')|$ and $S\cap U_0=\emptyset$, we have
		$|S\setminus V(Q)|
		\geq |S|-|V(Q)\setminus C_0|
		\geq |C_0|$.
		Hence there exists an injection
		\[
		\phi:C_0\longrightarrow S\setminus V(Q).
		\]
		Replacing every $x\in C_0$ by $\phi(x)$ transforms each new edge
		$\{x\}\cup A$ of $Q$ into an edge of $H$ since
		$A\in N(S)\subseteq N_H(\phi(x))$. Since every other edge of $Q$ is an edge of $G$, we obtain a copy of $F'$ in $G$. This contradiction implies that $G'$ is $\cF$-free and $|E(G')|\leq\operatorname{ex}_k(|U\cup V(M)|,\cF)\leq\operatorname{ex}_k(ks+k-1,\cF)$.
		
		By the Claim, every edge intersecting $I$ also intersects
		$V(M_2)$. Hence $\sum_{x\in I}d_G(x)\leq |I||V(M_2)|n^{k-2}=k|I||U_0|n^{k-2}.$
		By the definition of $G'$, 
		\begin{equation}
			\begin{split}
				|E(G')|
				&\geq  |E(G)|-\sum_{x\in U_0\cup I}d_G(x)+|U_0|\left(|N(S)|-|U_0|n^{k-2}\right)\\
				&\geq |E(G)|-\sum_{x\in I}d_G(x)-2k^2\sqrt{\varepsilon}|U_0|n^{k-1}+|U_0|\left(|N(S)|-|U_0|n^{k-2}\right)\\
				&\geq|E(G)|+|U_0|\left(|N(S)|-|U_0|n^{k-2}-2k^2\sqrt{\varepsilon}n^{k-1}
				-k|I|n^{k-2}\right)\\
				&\geq |E(G)|+|U_0|\left(\frac{1}{2\ell^{k-1}}-(3k^2+k)\sqrt{\varepsilon}\right)n^{k-1}>|E(G)|
			\end{split}
		\end{equation}
		since $\varepsilon$ is sufficiently small. Therefore
		$|E(G)|<|E(G')|\leq\operatorname{ex}_k(ks+k-1,\cF)$.
		This completes the proof.	
	\end{proof}

	\begin{proof}[Proof of Theorem~\ref{thm:intro-graph}]
		The lower bound follows by taking a maximum $F$-free graph on
		$2s+1$ vertices and adding $n-2s-1$ isolated vertices.
		
		For the upper bound, let $G$ be a maximum $F$-free graph on $n$
		vertices with $\nu(G)\le s$ and $\ell=\chi(F)-1$. Since  $T_2(2s+1,\ell)$ is
		$F$-free,
		$|E(G)|\ge\ex_2(2s+1,F)\ge t_2(2s+1,\ell)$.
		The inequality $s>n/2-\beta n$ implies $2s+1\geq n-2\beta n$. It follows that
		$t_2(2s+1,\ell)\ge t_2(n,\ell)-2\beta n^2$  for all sufficiently large $n$. Taking $\beta$
		small enough and applying Theorem \ref{erdos-stability}, we can derive that $G$ is $\eps$-close to $T_2(n,\ell)$, where $\eps$ is required by
		Lemma~\ref{lem:close-reduction}. Applying Lemma~\ref{lem:close-reduction} with $k=2$ yields
		$|E(G)|\le\ex_2(2s+1,F)$, as desired.
	\end{proof}
	
	\begin{proof}[Proof of Theorem~\ref{thm:intro-hypergraph}]
		For the lower bound, take a maximum
		$H_{\ell+1}^k$-free $k$-graph on $ks+k-1$ vertices together with $n-(ks+k-1)$
		isolated vertices. This gives
		$\ex_k(n,\{H_{\ell+1}^k,M_{s+1}^k\})
		\ge\ex_k(ks+k-1,H_{\ell+1}^k)$.
		
		For the upper bound, let $G$ be a maximum
		$H_{\ell+1}^k$-free $n$-vertex $k$-graph with $\nu(G)\le s$. By Theorem \ref{thm:weak-turan},
		$|E(G)|\ge t_k(ks+k-1,\ell)$.
		Since $s>n/k-\beta n$, we have $ks+k-1>(1-k\beta)n$. Hence for fixed
		$k,\ell$ and sufficiently large $n$,
		$t_k(ks+k-1,\ell)\ge t_k(n,\ell)-k\beta n^k$. Choose $\beta$ sufficiently small. Theorem \ref{thm:exp-stability} implies that $G$ is $\eps$-close to
		$T_k(n,\ell)$, where
		$\eps$ is required by Lemma~\ref{lem:close-reduction}.
		Thus $|E(G)|\le\ex_k(ks+k-1,H_{\ell+1}^k)$ by Lemma~\ref{lem:close-reduction}.
	\end{proof}
	
	\section{Proof of Theorem~\ref{thm:intro-boundary}}
	Theorem~\ref{thm:intro-hypergraph} determines $\ex_k(n,\{H_{\ell+1}^k,M_{s+1}^k\})$ for $\ell>k$. In contrast to the case $\ell>k$, the case $\ell=k$ has a different extremal graph. In this section, we deal with this case.
	For $0\leq x\leq \lceil n/k\rceil$, define
	$h_{n,k}(x)=x\,t_{k-1}(n-x,k-1)$.
	
	\begin{lemma}\label{lem:boundary-product}
		For every integer $x$ satisfying $0\le x<n/k$,
		$0\le h_{n,k}(x+1)-h_{n,k}(x)\le(n-kx)n^{k-2}$.
	\end{lemma}
	
	\begin{proof}
		Let $\{V_1,\ldots,V_{k-1}\}$ be a partition of $[n-x]$ such that $|V_i|=\floor{\frac{n-x+i-1}{k-1}}$ for each $i\in[k-1]$. Note that $|V_1|\leq |V_2|\leq \cdots\leq |V_{k-1}|$ and $t_{k-1}(n-x,k-1)=\prod_{i\in[k-1]}|V_i|$. By the definition of $t_{k-1}(n-x-1,k-1)$, one can see that $t_{k-1}(n-x-1,k-1)=(|V_{k-1}|-1)\prod_{i\in[k-2]}|V_i|$. Thus
		\begin{equation*}
			\begin{split}
				h_{n,k}(x+1)-h_{n,k}(x)=&(x+1)(|V_{k-1}|-1)\prod_{i\in[k-2]}|V_i|-x\prod_{i\in[k-1]}|V_i|\\
				=&\bigl((x+1)(|V_{k-1}|-1)-x|V_{k-1}|\bigr)\prod_{i\in[k-2]}|V_i|\\
				=&(|V_{k-1}|-x-1)\prod_{i\in[k-2]}|V_i|.
			\end{split}
		\end{equation*}
		Since $0\leq x<n/k$, we have
		$0\le |V_{k-1}|-x-1\le(n-kx)/(k-1)$ and the proof is done.
	\end{proof}
	
	\begin{proof}[Proof of Theorem~\ref{thm:intro-boundary}]
		The fact that the $k$-graph $H(n,k,s,k)$ is
		$H_{k+1}^k$-free, has matching number at most $s$, and contains
		$s\,t_{k-1}(n-s,k-1)$ edges  yields the lower bound of 	$\ex_k\bigl(n,\{H_{k+1}^k,M_{s+1}^k\}\bigr)$.

		For the upper bound, let $G$ be a maximum $n$-vertex
		$\{H_{k+1}^k,M_{s+1}^k\}$-free $k$-graph. Then $|E(G)|\ge s\,t_{k-1}(n-s,k-1)$.
		Let $\gamma=\frac{1}{1024k^2(2k)^{k-2}}$ and let $\varepsilon$ be a sufficiently small constant such that $0<\varepsilon<\left(\frac{\gamma}{4(2k)^k}\right)^2$.
		Let $\delta>0$ be the constant provided by Theorem~\ref{thm:exp-stability} with parameters
		$\varepsilon$ and $\ell=k$. 
		Let $\beta$ be a constant such that $0<\beta<\min\left\{2\delta,\frac{1}{2^kk^k}\right\}$. Recall that $\max\{s_0,n/k-\beta n\}<s\leq(n-k)/k$.
		Thus
		$$t_k(n,k)-|E(G)|
		\le t_k(n,k)-s\,t_{k-1}(n-s,k-1)
		\le\bigl(\ceil{n/k}-s\bigr)\ceil{n/k}^{k-1}
		<(\beta n^k)/2.$$
		By Theorem~\ref{thm:exp-stability}, there is a partition
		$\{W_1,\ldots,W_k\}$ of $[n]$ such that $|W_i|=\floor{\frac{n+i-1}{k}}$ for $i\in [k]$ and
		$|E(T_0)\setminus E(G)|\le\varepsilon n^k$, where $T_0$ is the complete $k$-partite $k$-graph  with parts $W_1,\ldots,W_k$.

		Recall that a pair $\{u,v\}$ is sparse  in $G$ if
		$d_G(u,v)\le k\binom{k+1}{2}n^{k-3}$ and dense in $G$ otherwise.
		Let $\mathcal S$ be the family of sparse pairs in $G$ joining distinct
		$W_i$ and $W_j$ for $i,j\in[k]$. For $\{u,v\}\in\mathcal S$,
		\[
		|N_{T_0}(u,v)\setminus N_G(u,v)|
		\ge\left(\frac{n}{2k}\right)^{k-2}
		-k\binom{k+1}{2}n^{k-3}
		\ge\frac12\left(\frac{n}{2k}\right)^{k-2}
		\]
		for sufficiently large $n$. Since every edge in $T_0$ contains
		$\binom{k}{2}$ subsets of size two,
		\[
		|E(T_0)\setminus E(G)|
		\ge\frac{1}{\binom{k}{2}}
		\sum_{\{u,v\}\in\mathcal S}
		|N_{T_0}(u,v)\setminus N_G(u,v)|
		\ge\frac{|\mathcal S|}{2\binom{k}{2}}
		\left(\frac{n}{2k}\right)^{k-2}.
		\]
		Together with $|E(T_0)\setminus E(G)|\le\varepsilon n^k$, we have
		\[
		|\mathcal S|
		\le2\binom{k}{2}(2k)^{k-2}\varepsilon n^2
		\le(2k)^k\varepsilon n^2.
		\]
		Let $F$ be the graph with vertex set $V(G)$ and edge set $\mathcal S$. Define
		$R=\{v:d_{F}(v)\ge\sqrt{\varepsilon}\,n\}$ and
		$U_i=W_i\setminus R$.
		Since
		$|R|\sqrt{\varepsilon}n\leq\sum_{v\in R}d_{F}(v)\leq 2|\mathcal S|$,
		the choice of $\varepsilon$ implies
		$|R|\le2(2k)^k\sqrt{\varepsilon}\,n<\frac{\gamma n}{2}$.
		It follows that $|U_i|\ge\frac n{2k}$ for $i\in[k]$.
		In particular, every vertex of $U_i$ has fewer than
		$\sqrt{\varepsilon}\,n$ sparse neighbours in the other parts in $G$.
		
		We now show that
		for $D\in E(G)$ and $i\in[k]$, $|D\cap U_i|\le1$.
		Indeed, suppose that $D$ contains distinct  $x,y\in U_{i_0}$ for some $i_0\in [k]$. Without loss of generality, we may assume that $i_0=k$. For
		each $j\ne k$, deleting from $U_j$ the vertices of $D$ and the sparse
		neighbours of $x$ or $y$, the remaining set has size at least
		$|U_j|/2$. By Lemma~\ref{lem:selection}, each
		remaining set has a vertex $z_i$ for $i\in [k-1]$ such that all pairs of $\{z_1,\ldots,z_{k-1}\}$ are dense. Thus $\{x,y,z_1,\ldots,z_{k-1}\}$ is a $(k+1)$-set whose only possibly sparse pair is $\{x,y\}$ and $D\cap \{x,y,z_1,\ldots,z_{k-1}\}=\{x,y\}$. By
		Lemma~\ref{obs:extension}, there is a copy of $H^k_{k+1}$ in $G$, a contradiction.
		
		Among all partitions $\mathcal P:=\{P_1,\ldots,P_k\}$ on $[n]$ with
		$U_i\subseteq P_i$, choose one maximizing
		\[
		\Phi(\mathcal P):=
		\sum_{D\in E(G)}|\{i\in [k]:D\cap P_i\ne\emptyset\}|.
		\]
		Since $U_i\subseteq P_i$,
		$\left||P_i|-\frac nk\right|\le |R|+1\le\gamma n$ for $i\in[k]$.
		Let $T$ be the complete $k$-partite $k$-graph defined on $P_1\cup \cdots\cup P_k$, and let
		$\mathcal B=E(G)\setminus E(T)$ and $\mathcal M=E(T)\setminus E(G)$.
		For $z\in R$, let $i(z)\in [k]$  such that $z\in P_{i(z)}$, and let
		\[
		R_0=\left\{z\in R:
		z\text{ has at least }\frac{|U_j|}{8}\text{ sparse neighbours belonging to }
		U_j\text{ in $G$ for some }j\ne i(z)\right\}.
		\]
		A pair $\{x,y\}\subseteq P_i$ is called a \emph{bad pair} if it is
		contained in an edge of $\mathcal B$ and
		$\{x,y\}\cap R_0\ne\emptyset$. Let $\mathcal Q$ be the family of all
		bad pairs.
		
		We  claim that every edge of $\mathcal B$ contains a bad pair.
		Indeed, for each $D\in\mathcal B$, there exists $i_0$ such that $|D\cap P_{i_0}|\geq 2$ since $D\notin E(T)$. Let $x,y$ be two distinct vertices such that $x,y\in D\cap P_{i_0}$.
		If $\{x,y\}\cap R_0\ne\emptyset$, then the proof of claim is done. So we assume that $x,y\notin R_0$. If both $x$ and $y$ have fewer than $|U_j|/8$ sparse neighbours
		in $U_j$ for all $j\ne i_0$, then more than
		$3|U_j|/4$ vertices of $U_j$ are dense neighbours of both $x$ and $y$ for every $j\ne i_0$.
		Applying Lemma \ref{obs:extension} and \ref{lem:selection}, we obtain a copy of $H_{k+1}^k$ in $G$,
		a contradiction.
		Hence one of $x,y$, say $x$, has at least $|U_j|/8$ sparse neighbours
		in some $U_j$ with $j\ne i_0$. Moreover, $x\notin U_{i_0}$, since otherwise
		$d_{\mathcal S}(x)\ge\frac{|U_j|}{8}
		\ge\frac{n}{16k}>\sqrt{\varepsilon}\,n$, a contradiction.
		Thus $x\in R_0$, a contradiction.
		
		For each $z\in R_0$, let $U_{j(z)}$ be the part in which $z$ has at least $|U_{j(z)}|/8$ sparse neighbours and
		let
		$S_z=\{u\in U_{j(z)}:\{z,u\}\text{ is sparse in } G\}$.
		Then $|S_z|\ge |U_{j(z)}|/8\ge n/(16k)$. For each $u\in S_z$, $\{z,u\}$ is a sparse pair. Thus 
		\[
		\begin{aligned}
			|\mathcal{F}_{z,u}|
			\ge
			\prod_{h\in[k]\setminus\{i(z),j(z)\}}|P_h|
			-|N_G(z,u)|
			\ge\frac12\left(\frac{n}{2k}\right)^{k-2}
		\end{aligned}
		\]
		for sufficiently large $n$, where  $\mathcal{F}_{z,u}=\{F\in\mathcal M:\{z,u\}\subseteq F\}$.
		For a fixed $z$, $\mathcal{F}_{z,u}$ with $u\in S_z$ are pairwise disjoint,
		and every edge of $\mathcal M$ is
		counted for at most $k$ choices of $z$. Hence
		\[
		\begin{aligned}
			|\mathcal M|\geq \frac1k\sum_{z\in R_0}\sum_{u\in S_z}|\mathcal{F}_{z,u}|
			\ge\frac1k\sum_{z\in R_0}|S_z|
			\frac12\left(\frac{n}{2k}\right)^{k-2}
			\ge\frac{|R_0|n^{k-1}}{32k^2(2k)^{k-2}}
			\geq 16\gamma|R_0|n^{k-1}.
		\end{aligned}
		\]
		
		We will prove that
		$|\mathcal B|\le|\mathcal M|/2$. Suppose that $|\mathcal B|>|\mathcal M|/2$. For each $z\in R_0$, let $\mathcal Q_z$ be the family of all bad pairs containing $z$ and 
		$\mathcal E_z=\{D\in\mathcal B:\binom{D}{2}\cap \mathcal Q_z\neq \emptyset\}$.
		Since every edge of $\mathcal B$ contains a bad pair, 
		\[
		\sum_{z\in R_0}|\mathcal E_z|
		\ge|\mathcal B|>\frac12|\mathcal M|
		>8\gamma|R_0|n^{k-1}.
		\]
		Hence there is $z\in R_0$, say $z\in P_i$, such that
		$|\mathcal E_z|>8\gamma n^{k-1}$. Since the number of edges belonging to $\mathcal E_z$ and intersecting $R\setminus \{z\}$  is at most
		$|R|n^{k-2}
		<\frac{\gamma}{2}n^{k-1}$, there is an edge $D\in \mathcal E_z$ containing a vertex $w$ such that 
		$w\in \cup_{i\in [k]}U_i$ and $\{z,w\}$ is a bad pair.
		
		Suppose that $z\in P_{i'}$. We claim that, for some $j\ne i'$, the vertex $z$ has fewer than
		$2\gamma n$ dense neighbours in $U_j$. Otherwise, after deleting the
		vertices of $D$ and the sparse neighbours of $w$, at least
		$\gamma n\ge\gamma|U_j|$ vertices remain in every $U_j$ for $j\ne i'$.
		Lemmas~\ref{obs:extension} and \ref{lem:selection} again give a copy of
		$H_{k+1}^k$, a contradiction.
		Thus we may assume that there exists $j'\neq i'$ such that $z$ has fewer than
		$2\gamma n$ dense neighbours in $U_{j'}$. It follows that the number of edges of $G$ containing $z$ and
		intersecting $P_{j'}$ is at most
		\[
		2\gamma n\binom{n-2}{k-2}
		+|U_{j'}|k\binom{k+1}{2}n^{k-3}
		+|R|\binom{n-2}{k-2}
		<3\gamma n^{k-1}.
		\]
		Let $\mathcal P'$ be obtained from $\mathcal P$ by moving $z$ from
		$P_{i'}$ to $P_{j'}$. Since $|\mathcal E_z|>8\gamma n^{k-1}$, more than $8\gamma n^{k-1}$ edges of $\mathcal B$ contain $z$ and another vertex of $P_{i'}$.
		Every such edge avoiding $P_{j'}$ increases its contribution to $\Phi$ by one, whereas any edge whose contribution to $\Phi$ decreases  contains $z$ and intersect $P_{j'}$.
		Therefore
		\[
		\Phi(\mathcal P')-\Phi(\mathcal P)
		>8\gamma n^{k-1}-2\cdot3\gamma n^{k-1}>0,
		\]
		a contradiction to the maximality of $\mathcal P$. This proves that $|\mathcal B|\le\frac12|\mathcal M|$.
		
		Let $a=\min_i|P_i|$. Since
		$|E(G)|=|E(T)|-|\mathcal M|+|\mathcal B|$,
		\begin{equation}\label{eq:boundary-main}
			|E(G)|\le |E(T)|-\frac12|\mathcal M|
			\le h_{n,k}(a)-\frac12|\mathcal M|.
		\end{equation}
		For the case $a\le s$, Lemma~\ref{lem:boundary-product} gives
		$|E(G)|\le h_{n,k}(a)\le h_{n,k}(s)$ and the proof is done.
		So we assume that $a>s$ in the following proof. Note that
		$0<a-s\le n/k-s<\beta n$.  Let $p_i=|P_i|$ for $i\in[k]$, and relabel the parts such that
		$p_1=a$.
		Choose  a matching $M_0$ of size $a$ in $T$
		that covers $P_1$ uniformly at random. Since $\nu(G)\le s$, at most $s$ edges of
		$M_0$ belong to $G$, and hence
		$|M_0\cap\mathcal M|\ge a-s$. Moreover, every edge of $T$
		belongs to $M_0$ with probability
		$1/\prod_{i=2}^k p_i$. Therefore
		\[
		a-s
		\le \mathbb E|M_0\cap\mathcal M|
		=\frac{|\mathcal M|}{\prod_{i=2}^k p_i},
		\]
		and consequently
		\[
		|\mathcal M|
		\ge(a-s)\prod_{i=2}^k p_i
		\ge(a-s)a^{k-1}\geq (a-s)s^{k-1}
		>(a-s)(1/k-\beta)^{k-1}n^{k-1}.
		\]
		By Lemma~\ref{lem:boundary-product},
		\[
		\begin{aligned}
			h_{n,k}(a)-h_{n,k}(s)
			=\sum_{x=s}^{a-1}\bigl(h_{n,k}(x+1)-h_{n,k}(x)\bigr)
			\le\sum_{x=s}^{a-1}(n-kx)n^{k-2}
			\le k\beta(a-s)n^{k-1}.
		\end{aligned}
		\]
		Combining  with inequality \eqref{eq:boundary-main}, we obtain
		\[
		|E(G)|\leq h_{n,k}(a)-\frac{(1/k-\beta)^{k-1}}2
		(a-s)n^{k-1} \le h_{n,k}(s)-
		\left(\frac{(1/k-\beta)^{k-1}}2-k\beta\right)
		(a-s)n^{k-1}
		\le h_{n,k}(s)
		\]
		where the last inequality follows from the choice of $\beta$. This completes the proof.
	\end{proof}

	\section{Proof of Theorems~\ref{thm:small-s} and \ref{thm:small-expansion}}
Throughout this section, define
$c_0=\frac{1}{8(k-1)^{k-2}(\ell-1)^3}$.


\begin{lemma}\label{lem:numerical}
	Let $n,\ell,k,s$ be integers such that $\ell>k\geq 3$. If $1\le s<c_0n$, then	
	\begin{equation}\label{eq:23}
		\left\{
	\begin{aligned}
		&t_{k-1}(n-s,\ell-1) >(k^2+1)s\binom{n-1}{k-2},\\ 
		&t_{k-1}(n-s,\ell-1)-(k+2)s\binom{n-1}{k-2}>t_{k-1}(n-2s,\ell-2).\\  
	\end{aligned}
		\right.
	\end{equation}
\end{lemma}

The proof of Lemma \ref{lem:numerical} is given in Appendix~\ref{app:numerical-proofs}.



The following lemma will be used to prove Theorems~\ref{thm:small-s} and \ref{thm:small-expansion}. In fact, in order to prove $\ex_k\bigl(n,\cK_{\ell+1}^k\cup\{M_{s+1}^k\}\bigr)\leq s\,t_{k-1}(n-s,\ell-1)$ (resp. $\ex_k\bigl(n,\{H_{\ell+1}^k\cup M_{s+1}^k\}\bigr)\leq s\,t_{k-1}(n-s,\ell-1)$), it suffices to prove that any $\cK_{\ell+1}^k$-free (resp.~$H_{\ell+1}^k$-free) $k$-graph $H$ with $\nu(H)\leq s$ satisfies the condition of Lemma \ref{lem:budget} (ii).
\begin{lemma}\label{lem:budget}
	Let $n, \ell, k,s$ be positive integers such that $n\geq \ell\geq k\geq 3$. Suppose that $F$ is an $n$-vertex $k$-graph with $\nu(F)\le s$, and $V_2:=\{v\in V(F):d_F(v)\le ks\binom{n-1}{k-2}\}$. Then the following statements hold:
	\begin{enumerate}[label=(\roman*)]
		\item
		$|V(F)\setminus V_2|+\nu(F[V_2])\le s$, and
		
		\item if $s<c_0n$ and there is an $X\subseteq V(F)\setminus V_2$ such that
		$d_F(v)<t_{k-1}(n-s,\ell-1)$ for every $v\in V(F)\setminus(X\cup V_2)$ and
		$\sum_{x\in X}d_{F[X\cup V_2]}(x)
		\le |X|t_{k-1}(|V_2|,\ell-1)$,
		then $|E(F)|\le st_{k-1}(n-s,\ell-1)$.
	\end{enumerate}
\end{lemma}

\begin{proof}
	(i) Let $m=\nu(F[V_2])$ and $M$ be a matching  in $F[V_2]$ of size $m$. Suppose, to the contrary, that $|V(F)\setminus V_2|+m\ge s+1$. Let $h=s+1-m$. Then $|V(F)\setminus V_2|\ge h$.
	Let $\{z_1,\ldots,z_h\}$ be a set of $h$ distinct vertices in $V(F)\setminus V_2$.
	We will greedily choose pairwise disjoint edges $e_i\in E(F)$ such that $z_i\in e_i$ and $e_i\cap V(M)=\emptyset$.
	For  $z_1$, we choose an edge $e_1\in E(F)$ such that $z_1\in e_1$ and $e_1\cap (V(M)\cup \{z_2,\ldots z_h\})=\emptyset$.
	Suppose that we have found a matching $M'=\{e_1,\ldots,e_{t-1}\}$ such that $z_i\in e_i$ for $i\in [t-1]$ and $V(M')\cap (V(M)\cup \{z_t,\ldots z_h\})=\emptyset$.
	Since $d_F(z_t)>ks\binom{n-1}{k-2}\geq (km+k(t-1)+(h-t))\binom{n-1}{k-2}$, there exists an edge $e_t$ such that $z_t\in e_t$ and $e_t\cap (V(M)\cup V(M')\cup \{z_{t+1},\ldots z_h\})=\emptyset$. Set $M':= M'\cup \{e_t\}$.
	Continuing this process, we can find a matching $M'$ of size $h$. Then $M\cup M'$ is a matching of size $s+1$, a contradiction.
	
	(ii)
	Let $t=|X|$, $a=|(V(F)\setminus V_2)\setminus X|$, and $h=s-|V(F)\setminus V_2|$. Then $h=s-t-a$. By (i), we have $h\ge0$ and $\nu(F[V_2])\le h$. Since a maximum matching in $F[V_2]$ intersects
	every edge of $F[V_2]$ and every vertex of $V_2$
	has degree at most $ks\binom{n-1}{k-2}$, we have
	$|E(F[V_2])|\le kh\left(ks\binom{n-1}{k-2}\right)=k^2sh\binom{n-1}{k-2}$.
	Note that $|V_2|=n-a-t=(n-s)+h$ and $t\le s$.  Thus
	\begin{align*}
		|E(F[X\cup V_2])|
		&\le\sum_{x\in X}d_{F[X\cup V_2]}(x)+|E(F[V_2])|\\
		&\le t\,t_{k-1}(|V_2|,\ell-1)+|E(F[V_2])|\\
		&\le t(t_{k-1}(n-s,\ell-1)+h\binom{n-1}{k-2})+k^2sh\binom{n-1}{k-2}.
	\end{align*}
	By Lemma \ref{lem:numerical}, $(t+k^2s)\binom{n-1}{k-2}\le(k^2+1)s\binom{n-1}{k-2}<t_{k-1}(n-s,\ell-1)$.  Therefore
	\begin{align*}
		|E(F)|&\le at_{k-1}(n-s,\ell-1)+|E(F[X\cup V_2])|\\
		&\le(a+t)t_{k-1}(n-s,\ell-1)+h(t+k^2s)\binom{n-1}{k-2}\\
		&\le(a+t)t_{k-1}(n-s,\ell-1)+ht_{k-1}(n-s,\ell-1)=st_{k-1}(n-s,\ell-1).
	\end{align*}
\end{proof}

\subsection{Proof of Theorem~\ref{thm:small-s}}

\begin{observation}\label{obs:weak-link}
	If a $k$-graph $H$ is $\cK_{\ell+1}^k$-free, then $L_H(v)$ is
	$\cK_\ell^{k-1}$-free for every $v\in V(H)$.
\end{observation}

\begin{proof}
	Suppose that $L_H(v)$ contains a member of $\cK_\ell^{k-1}$ with core $C$. Adding $v$ to every edge of this subgraph yields a member of $\cK_{\ell+1}^k$ with core $C\cup\{v\}$, a contradiction.
\end{proof}

\begin{lemma}\label{prop:weak-structure}
	Let $n, \ell, k, s$ be integers such that $\ell\geq k\geq 3$ and $1\leq s<c_0n$. Suppose that $F$ is an $n$-vertex $k$-graph, $V_0:=\{v:d_F(v)\ge t_{k-1}(n-s,\ell-1)\}$ and $V_2:=\{v\in V(F):d_F(v)\le ks\binom{n-1}{k-2}\}$. If $F$ is $\cK_{\ell+1}^k$-free and $\nu(F)\le s$, then
	$\sum_{x\in V_0}d_{F[V_0\cup V_2]}(x)
	\le |V_0|\,t_{k-1}(|V_2|,\ell-1)$.
\end{lemma}

\begin{proof}
	We first prove that $V_0$ is strongly independent. Suppose, to the contrary, that there is an edge
	$e\in E(F)$ containing two distinct vertices $x,y\in V_0$. By
	Observation~\ref{obs:weak-link}, $L_F(x)$ is
	$\cK_\ell^{k-1}$-free. Define $V(L_F(u))=\cup_{f\in E(L_F(u))} f$ for every $u\in V(F)$. If $|V(L_F(x))|<n-s$,  Theorem \ref{thm:weak-turan} implies that
	$	d_F(x)\le t_{k-1}(n-s-1,\ell-1)<t_{k-1}(n-s,\ell-1)$,
	contradicting $x\in V_0$. Thus $|V(L_F(x))|\ge n-s$, and similarly
	$|V(L_F(y))|\ge n-s$. Hence
	$|V(L_F(x))\cap V(L_F(y))|\ge n-2s$.
	Let $W\subseteq(V(L_F(x))\cap V(L_F(y)))\setminus e$ be a set of size $n-2s-k$. By Lemma \ref{lem:numerical},
	\[
	|E(L_F(x)[W])|\ge t_{k-1}(n-s,\ell-1)-(2s+k)\binom{n-1}{k-2}
	>t_{k-1}(|W|,\ell-2).
	\]
	 for $\ell>k$ and Theorem~\ref{thm:weak-turan} implies that there is a member of
	$\cK_{\ell-1}^{k-1}$ in $L_F(x)[W]$. For the case $\ell=k$, $|E(L_F(x)[W])|\ge t_{k-1}(n-s,\ell-1)-(2s+k)\binom{n-1}{k-2}>0$, so any edge of $L_F(x)[W]$ is
	a member of $\cK_{k-1}^{k-1}$. In either case, let $G\in\cK_{\ell-1}^{k-1}$ be a subgraph
	of $L_F(x)[W]$ with core $C$.
	Every pair in $C$ and every pair $\{x,u\}$ with $u\in C$ is covered by an edge of $F$ containing $x$.
	Since $C\subseteq W\subseteq V(L_F(y))$, every pair $\{y,u\}$ with $u\in C$ is also covered by an edge of $F$. Moreover, $\{x,y\}$ is covered by $e$. Hence $C\cup\{x,y\}$ is the core of a member of $\cK_{\ell+1}^k$ in $F$, a contradiction.
	Thus $V_0$ is strongly independent.
	
	Now let $x\in V_0$. Since $V_0$ is strongly independent, every edge of
	$F[V_0\cup V_2]$ containing $x$ has its remaining $k-1$ vertices in $V_2$.
	Recall that $L_F(x)[V_2]$ is $\cK_\ell^{k-1}$-free for every $x\in V_0$. Hence
	$d_{F[V_0\cup V_2]}(x)=|E(L_F(x)[V_2])|
	\le t_{k-1}(|V_2|,\ell-1)$.
	Summing over $x\in V_0$ proves the lemma.
\end{proof}

\begin{proof}[Proof of Theorem~\ref{thm:small-s}]
	The lower bound is given by the $k$-graph $H(n,\ell,s,k)$. For the upper bound, let $F$ be a
	$\cK_{\ell+1}^k$-free $n$-vertex $k$-graph with $\nu(F)\le s$. By Lemma~\ref{lem:budget}(ii) and Lemma~\ref{prop:weak-structure}, we can derive that $|E(F)|\le st_{k-1}(n-s,\ell-1)$. 
\end{proof}

\subsection{Proof of Theorem~\ref{thm:small-expansion}}

\begin{lemma}\label{lem:exp-numerical}
	Let $n,\ell,k,s$ be integers such that $\ell>k\geq 3$, and let
	$\rho,\eta$ be positives such that
	$\rho<\min\left\{
	\frac{k-2}{64(k-1)^{k-2}(\ell-1)^2},
	\frac{c_0}{8}
	\right\}$ and
	$\eta<\frac{\rho}{2^{k+1}(k-1)^{k-1}}$.
	For $1\leq s<c_0n$ and sufficiently large $n$, the following holds.
	\begin{equation}\label{eq:25}
		\left\{
		\begin{aligned}
			&t_{k-1}(n-s,\ell-1)-\eta n^{k-1}>(k^2+1)s\binom{n-1}{k-2},\\ 
			&t_{k-1}(n-s,\ell-1)-\eta n^{k-1}>t_{k-1}(n-s-\rho n,\ell-1),\\ 
			& t_{k-1}(n-s,\ell-1)-\eta n^{k-1}-(2s+2\rho n+k)\binom{n-1}{k-2}>t_{k-1}(n-2s-2\rho n-k,\ell-2)\\
		\end{aligned}
		\right.
	\end{equation}
\end{lemma}

The proof of Lemma \ref{lem:exp-numerical} is given in Appendix~\ref{app:numerical-proofs}.

The following lemma shows that if $H$ is an $n$-vertex $H_{\ell+1}^k$-free $k$-graph with $\nu(H)\le s$ and $|E(H)|\ge st_{k-1}(n-s,\ell-1)$, then the structure of $H$ is close to $H(n,\ell,s,k)$.
\begin{lemma}\label{lem:exp-structure}
	Let $n,s,\ell,k$ be integers such that $\ell\geq k\geq 3$. There exist positive constants
	$\rho,\varepsilon,\xi,\eta$ and integers $s_0=s_0(k,\ell)$ and
	$n_0=n_0(k,\ell)$ such that the following holds.
	Let $H$ be an $n$-vertex $H_{\ell+1}^k$-free $k$-graph with $\nu(H)\le s$, where $s_0\le s<c_0n$ and $n\ge n_0$.
	If $|E(H)|\ge st_{k-1}(n-s,\ell-1)$, then
	there exist a set $X\subseteq V(H)$ and pairwise disjoint sets
	$U_1,\ldots,U_{\ell-1}\subseteq V(H)\setminus X$ such that
	$|X|\ge s-\frac{k\binom{\ell+1}{2}}{\eta}$,
	$d_H(x)<t_{k-1}(n-s,\ell-1)$ for every $x\notin X$ and
	$d_H(x)\ge t_{k-1}(n-s,\ell-1)-\eta n^{k-1}$ for every $x\in X$.
	Moreover:
	\begin{enumerate}[label=(\roman*)]
		\item $|U_i|\ge(1-\varepsilon)\frac{n-s}{\ell-1}$ for $i\in [\ell-1]$;\label{prop1}
		\item $|e\cap U_i|\le1$ for every $e\in E(H)$ and every $i\in [\ell-1]$;\label{prop2}
		\item $X$ is a 	strongly independent set;\label{prop3}
		\item every $x\in X$ has at most $(4k\ell)^k\rho n$ sparse neighbours in
		$\bigcup_{i\in [\ell-1]}U_i$;\label{prop4}
		\item every $u\in\bigcup_iU_i$ has at most $(4k\ell)^{k/2}\sqrt\rho\,|X|$ sparse neighbours in $X$;\label{prop5}
		\item for every $i\in[\ell-1]$ and $u\in U_i$, at most $\xi^{1/3}n$ vertices of $\bigcup_{j\ne i}U_j$ are sparse neighbours of $u$.\label{prop6}
	\end{enumerate}
\end{lemma}

\begin{proof}
	Let $\varepsilon,\rho,\xi,\eta$ be constants such that $0<\varepsilon<1/(256k^3(4\ell)^{k-2})$, 
	$0<\rho<\frac{\varepsilon^2}{16(\ell-1)^2(4k\ell)^k}$, $0<\xi<\rho^3/(4k\ell)^{3k}$ and $0<\eta\ll\xi$.
	Assume $s_0$ and $n_0$ are sufficiently large.
	Let $G$ be a spanning subgraph of $H$ obtained by deleting every edge that contains a sparse pair in $H$. Since there are at most $\binom n2$ pairs of vertices of $H$ and each sparse pair is contained in at most $k\binom{\ell+1}{2}n^{k-3}$ edges of $H$, we have
	$|E(H)\setminus E(G)|\le\binom n2 k\binom{\ell+1}{2}n^{k-3}$.
	Therefore
	\begin{equation}
		|E(G)|\ge |E(H)|-\binom n2 k\binom{\ell+1}{2}n^{k-3}
		\ge st_{k-1}(n-s,\ell-1)-k\binom{\ell+1}{2}n^{k-1}. \label{eq:52}
	\end{equation}
	Let
	$Y=\{y\in V(G):d_G(y)\ge t_{k-1}(n-s,\ell-1)-\eta n^{k-1}\}$.
	By Lemma~\ref{obs:extension}, it is not difficult to see that $G$ is $\cK_{\ell+1}^{k}$-free.
	Thus $L_G(y)$ is $\cK_\ell^{k-1}$-free for every $y\in Y$ by Observation~\ref{obs:weak-link} . If
	$|V(L_G(y))|<n-s-\rho n$, then $d_G(y)\leq t_{k-1}(n-s-\rho n,\ell-1)< t_{k-1}(n-s,\ell-1)-\eta n^{k-1}$ by Lemma \ref{lem:exp-numerical} and Theorem \ref{thm:weak-turan}, contradicting $y\in Y$. Thus for each $y\in Y$,
	\begin{equation}
		|V(L_G(y))|\ge n-s-\rho n. \label{eq:53}
	\end{equation}
	
	We claim that $Y$ is a strongly independent set in $H$. Otherwise, there is an edge
	$e\in E(H)$ containing two vertices $x,y\in Y$. By \eqref{eq:53}, let
	$W\subseteq(V(L_G(x))\cap V(L_G(y)))\setminus e$ be a set such that
	$|W|=\lfloor n-2s-2\rho n-k\rfloor$.
	By Lemma \ref{lem:exp-numerical}, we have
	\[
	|E(L_G(x)[W])|\ge t_{k-1}(n-s,\ell-1)-\eta n^{k-1}-(2s+2\rho n+k)\binom{n-1}{k-2}
	>t_{k-1}(|W|,\ell-2).
	\]
	for $\ell>k$. Then Theorem~\ref{thm:weak-turan} implies that there is a member of
	$\cK_{\ell-1}^{k-1}$ in $L_G(x)[W]$ for $\ell>k$. For the case $\ell=k$, $|E(L_G(x)[W])|\ge t_{k-1}(n-s,\ell-1)-\eta n^{k-1}-(2s+2\rho n+k)\binom{n-1}{k-2}
	>0$ and an edge of $L_G(x)[W]$ is
	a member of $\cK_{k-1}^{k-1}$. In either case, let $C$ denote its core. Every pair in
	$C\cup\{x,y\}$ except possibly $\{x,y\}$ is contained in an edge of $G$ and is
	therefore dense in $H$. Notice that $\{x,y\}$ is covered by $e$, and
	$e\cap(C\cup\{x,y\})=\{x,y\}$.
	By Lemma~\ref{obs:extension}, $H$ has a copy of $H_{\ell+1}^k$, a
	contradiction.
	
	Let
	$W_1=\{v:ks\binom{n-1}{k-2}<d_G(v)<t_{k-1}(n-s,\ell-1)-\eta n^{k-1}\}$ and
	$W_2=V(G)\setminus(Y\cup W_1)$.
	By Lemma~\ref{lem:budget}(i), one can see that $|Y|+|W_1|+\nu(G[W_2])\le s$.
	Let $h=s-|Y|-|W_1|$. Then $h\ge0$. Using similar arguments as in the proof of Lemma~\ref{lem:budget}(ii),
	we can derive that
	$|E(G)|\leq|W_1|(t_{k-1}(n-s,\ell-1)-\eta n^{k-1})+|Y|(t_{k-1}(n-s,\ell-1)+h\binom{n-1}{k-2})+k^2sh\binom{n-1}{k-2}$.
	Since $(|Y|+k^2s)\binom{n-1}{k-2}\le(k^2+1)s\binom{n-1}{k-2}<t_{k-1}(n-s,\ell-1)-\eta n^{k-1}$ by Lemma \ref{lem:exp-numerical}, we have
	$|E(G)|\leq st_{k-1}(n-s,\ell-1)-(s-|Y|)\eta n^{k-1}$.
	Combining with inequality \eqref{eq:52}, we have
	$s-|Y|\le\frac{k\binom{\ell+1}{2}}{\eta}$.
	By the choice of $s_0$,
	$|Y|\ge s_0-k\binom{\ell+1}{2}/\eta\ge\binom{\ell+1}{2}$. Let
	$Y'\subseteq Y$ such that $|Y'|=\binom{\ell+1}{2}$, and let $y_0\in Y'$.
	Define $V_0=\{x\in V(H):d_H(x)\ge t_{k-1}(n-s,\ell-1)\}$ and $X=V_0\cup Y$. Then $|X|\geq s-\frac{k\binom{\ell+1}{2}}{\eta}$,
	$d_H(x)<t_{k-1}(n-s,\ell-1)$ for every $x\notin X$, and	$d_H(x)\ge t_{k-1}(n-s,\ell-1)-\eta n^{k-1}$ for every $x\in X$  as required.
	
	Recall that $d_H(x)>ks\binom{n-1}{k-2}$ for every $x\in Y\cup W_1$. By Lemma~\ref{lem:budget}(i), we have
	$|(Y\cup W_1)\cup V_0|\le s$, and hence $|(W_1\cup V_0)\setminus Y|\le\frac{k\binom{\ell+1}{2}}{\eta}\le\eta n.$
	Let $Q=W_2\setminus V_0$. Recall that for every $y\in Y'$, $L_G(y)$ is $\cK_\ell^{k-1}$-free
	and
	\begin{equation}
		\begin{split}
			|E(L_G(y)[Q])|
			&\ge t_{k-1}(n-s,\ell-1)-\eta n^{k-1}-|(W_1\cup V_0)\setminus Y|n^{k-2}\\
			&\ge t_{k-1}(n-s,\ell-1)-2\eta n^{k-1}\ge t_{k-1}(|Q|,\ell-1)-3\eta n^{k-1}.
		\end{split}
	\end{equation}
	The last inequality holds since
	$n-s\le|Q|\le n-s+k\binom{\ell+1}{2}/\eta$ and $n$ is sufficiently large.
	Since $H_\ell^{k-1}\in \cK_\ell^{k-1}$, applying Theorem~\ref{thm:exp-stability} with parameters $k-1$ and $\ell-1$ when $k>3$, and Theorem~\ref{erdos-stability} when $k=3$, we conclude that $L_G(y)[Q]$ is $\xi$-close to $T_{k-1}(|Q|,\ell-1)$.
	Thus there is a partition $\{V_1^y, \ldots, V_{\ell-1}^y\}$ of $Q$ satisfying $|V_i^y|=\lfloor(|Q|+i-1)/(\ell-1)\rfloor$ for $i\in [\ell-1]$ such that $L_G(y)[Q]$ can be transformed into the complete $(\ell-1)$-partite $(k-1)$-graph $K_y$ on $V_1^y\cup \cdots\cup V_{\ell-1}^y$ by adding and deleting at most $\xi n^{k-1}$ edges.
	Since $s\leq c_0n$, we may assume that
	$|V_i^y|\ge n/(2\ell)$ for every $y\in Y'$ and
	$i\in[\ell-1]$.
	
	For $y\in Y'$, let
	$A_{y}=\{u\in Q:\{y,u\}\text{ is sparse in }H\}$.
	Recall that every edge of $G$ contains only dense pairs of $H$. Hence
	$u\in A_{y}$ implies $d_{L_G(y)[Q]}(u)=0$. Since every edge of $K_y$ containing such a vertex is missing from $L_G(y)[Q]$ and each missing edge is counted at most $k-1$ times, we have
	\[	\begin{aligned}
		\xi n^{k-1}
		\ge |E(K_{y})\setminus E(L_G(y)[Q])|
		\ge \frac{|A_{y}|}{k-1}
		\binom{\ell-2}{k-2}\left(\frac{n}{2\ell}\right)^{k-2}.
	\end{aligned}
	\]
	Consequently,
	\begin{equation}
		|A_{y}|\le (4k\ell)^k\xi n.  \label{eq:Ay}
	\end{equation}
	
	Let $y_0\in Y'$. With respect to the fixed partition $\{V_1^{y_0},\cdots, V_{\ell-1}^{y_0}\}$, let
	\[
	\mathcal S_0=\{\{u,v\}:u\in V_i^{y_0},\ v\in V_j^{y_0},\ i\ne j, \ i,j\in [\ell-1],
	\ \{u,v\}\text{ is sparse in }H\}.
	\]
	If $\{u,v\}\in\mathcal S_0$, then no edge of $L_G(y_0)[Q]$ contains $u$ and $v$.
	Every such pair is contained in at least
	$\binom{\ell-3}{k-3}(n/(2\ell))^{k-3}$ edges of $K_{y_0}$ and
	each missing edge contains at most $\binom{k-1}{2}$ pairs. Therefore
	\[
	\begin{aligned}
		\xi n^{k-1}
		\ge |E(K_{y_0})\setminus E(L_G(y_0)[Q])|
		\ge \frac{|\mathcal S_{0}|}{\binom{k-1}{2}}
		\binom{\ell-3}{k-3}\left(\frac{n}{2\ell}\right)^{k-3}.
	\end{aligned}
	\]
	It follows that
	$|\mathcal S_0|\le (4k\ell)^k\xi n^2$.
	
	For $u\in V_i^{y_0}$, $ i\in [\ell-1]$, let
	$d_{\mathcal S_0}(u)$ denote the number of vertices in
	$\bigcup_{j\ne i}V_j^{y_0}$ each of which together with $u$ forms a sparse pair in $\mathcal S_0$, and let
	$B_0=\{u\in Q:d_{\mathcal S_0}(u)>\xi^{1/3}n\}$.
	Double counting the pairs in $\mathcal S_0$ yields $|B_0|\xi^{1/3}n\leq 2|\mathcal S_0|$. Thus
	\begin{equation}
		|B_0|\le 2(4k\ell)^k\xi^{2/3}n. \label{eq:B0}
	\end{equation}
	For $i\in[\ell-1]$, let
	$U_i^0=V_i^{y_0}\setminus
	\left(B_0\cup\bigcup_{y\in Y'}A_y\right)$.
	By \eqref{eq:Ay}, \eqref{eq:B0}
	and by the choice of $\rho$ and $\xi$,
	\begin{equation}
		\min_i|U_i^0|\ge(1-\varepsilon/2)\frac{n-s}{\ell-1}
		\mbox{ and }
		\left|Q\setminus\bigcup_iU_i^0\right|\le\rho n. \label{eq:54}
	\end{equation}
	By the definition of $A_y$, every vertex of $\bigcup_iU_i^0$ is a dense neighbour of every vertex of $Y'$ in $H$. By the definition of $B_0$,
	each $u\in U_i^0$ has at most $\xi^{1/3}n$ sparse neighbours in
	$\bigcup_{j\ne i}U_j^0$. In particular, the family of sparse pairs
	joining distinct $U_i^0$'s is a subfamily of $\mathcal S_0$ and has at
	most $(4k\ell)^k\xi n^2$ members.
	
	\vspace{0.15cm}	
	{\bf Claim 1.} $|e\cap U_i^0|\le1$ for every $e\in E(H)$ and
	every $i \in [\ell-1]$.
	\vspace{0.15cm}
	
	Otherwise, after relabelling, we suppose that there is an edge $e$ that contains
	distinct $u_0,u_1\in U_1^0$. Since
	$|Y'|=\binom{\ell+1}{2}>k$, choose $y\in Y'\setminus e$. For
	$j=2,\ldots,\ell-1$, let
	$Z_j=\{u\in U_j^0\setminus e: \{u_0,u\}\text{ and } \{u_1,u\}
	\text{ are dense  in $H$} \}$.
	Thus $|Z_j|\ge|U_j^0|-2\xi^{1/3}n-k\ge|U_j^0|/2$. For the case $\ell=3$, let $z_2\in Z_2$. For the case $\ell>3$, applying
	Lemma~\ref{lem:selection} to the family of sparse pairs joining
	distinct $U_j^0$'s, there is a set $\{z_j\in Z_j: 2\le j\le \ell-1\}$ in which
	every pair  is dense in $H$. In either case, let $C:=\{y,u_0,u_1\}\cup\{z_j:j\ne1\}$. All pairs of
	$C$ are dense except $\{u_0,u_1\}$ since the
	vertex $y$ is  a dense neighbour of each of $u_0,u_1,z_2,\ldots,z_{\ell-1}$. Note that
	$e\cap C=\{u_0,u_1\}$. Lemma~\ref{obs:extension} gives a copy of
	$H_{\ell+1}^k$ in $H$, a contradiction. This completes the proof of Claim 1.
	
	Let $K^0$ be the complete $(\ell-1)$-partite $(k-1)$-graph on
	$U_1^0\cup \cdots \cup U_{\ell-1}^0$. We prove that for each $x\in X$ and $i\in[\ell-1]$,
	\begin{equation}\label{eq:coarse}
		|\{u\in U_i^0:\{x,u\}\text{ is sparse  in $H$ }\}|<\frac13|U_i^0|.
	\end{equation}
	Recall that every $x\in X$ satisfies
	$d_H(x)\ge t_{k-1}(n-s,\ell-1)-\eta n^{k-1}$. Moreover,
	$V(H)\setminus Q$ is contained in
	$Y\cup\bigl((W_1\cup V_0)\setminus Y\bigr)$,
	where $|Y|\le s$ and
	$|(W_1\cup V_0)\setminus Y|\le k\binom{\ell+1}{2}/\eta$. Together with
	\eqref{eq:54}, we can derive that
	\[
	|E(L_H(x)[\textstyle\bigcup_iU_i^0])|
	\ge t_{k-1}(n-s,\ell-1)-\eta n^{k-1}
	-\left(s+\rho n+\frac{k\binom{\ell+1}{2}}{\eta}\right)\binom{n-1}{k-2}.
	\]
	Since $s<c_0n$,
	$|E(L_H(x)[\textstyle\bigcup_iU_i^0])|
	\ge t_{k-1}(n-s,\ell-1)-(c_0+\rho+2\eta)n^{k-1}.$
	Note that $L_H(x)[\textstyle\bigcup_iU_i^0]$ is a subgraph of $K^0$ and
	$|E(K^0)|\le t_{k-1}(|Q|,\ell-1)\le t_{k-1}(n-s,\ell-1)+\eta n^{k-1}$
	for sufficiently large $n$. Since $\rho+3\eta<c_0$, we have
	\begin{equation}\label{eq:coarse-missing}
		|E(K^0)\setminus E(L_H(x)[\textstyle\bigcup_iU_i^0])|<2c_0n^{k-1}.
	\end{equation}
	
	Let $m=\min_i|U_i^0|$. By \eqref{eq:54},
	$m^{k-1}\ge\frac12\left(\frac{n}{\ell-1}\right)^{k-1}$.
	Suppose that \eqref{eq:coarse} fails for some $x$ and $i$. Let
	$S_i=\{u\in U_i^0: \{x,u\}\text{ is sparse}\}$. Then $|S_i|\geq\frac13|U_i^0|\geq \frac13 m $. For each $u\in S_i$,
	$d_{K^0}(u)\ge\binom{\ell-2}{k-2}m^{k-2}$ and $d_{L_H(x)}(u)\le k\binom{\ell+1}{2}n^{k-3}$.
	Since no edge of $K^0$ contains distinct vertices of $U_i^0$,
	\begin{align*}
		|E(K^0)\setminus E(L_H(x)[\textstyle\bigcup_iU_i^0])|
		\ge |S_i|\left(\binom{\ell-2}{k-2}m^{k-2}-k\binom{\ell+1}{2}n^{k-3}\right)
		\ge\frac14\binom{\ell-2}{k-2}m^{k-1}
		\ge\frac18\frac{\binom{\ell-2}{k-2}}
		{ (\ell-1)^{k-1}}n^{k-1}.
	\end{align*}
	As
	$\binom{\ell-2}{k-2}\ge
	\left(\frac{\ell-1}{k-1}\right)^{k-2}$,
	the last expression is at least
	$\frac{n^{k-1}}{8(k-1)^{k-2}(\ell-1)}
	=(\ell-1)^2c_0n^{k-1}
	\ge4c_0n^{k-1}$,
	contradicting \eqref{eq:coarse-missing}. This proves \eqref{eq:coarse}.
	
	\vspace{0.15cm}
	{\bf Claim 2.} $X$ is a strongly independent set in $H$, and then (iii) follows.
	\vspace{0.15cm}
	
	Suppose, to the contrary,  that there is an edge
	$e$ of $H$ that contains distinct $a,b\in X$. For each $i\in[\ell-1]$, let
	$Z_i=\{u\in U_i^0\setminus e:
	\{a,u\}\text{ and }\{b,u\}\text{ are dense}\}$.
	By \eqref{eq:coarse}, $|Z_i|\ge|U_i^0|/3-k\ge|U_i^0|/4$.
	Applying Lemma~\ref{lem:selection} to the family of sparse pairs
	joining distinct $U_i^0$'s, there is a set $\{u_i\in Z_i: i\in[\ell-1]\}$ in which every
	pair  is dense in $H$. Thus
	$\{a,b,u_1,\ldots,u_{\ell-1}\}$ has all pairs dense except possibly $\{a,b\}$,
	while $e$ intersects this set exactly in $\{a,b\}$. Lemma~\ref{obs:extension}
	yields a copy of $H_{\ell+1}^k$  in $H$, a contradiction. This completes the proof of Claim 2.
	
	\vspace{0.15cm}
	{\bf Claim 3.} For each $x\in X$, there are at most $(4k\ell)^k\rho n$ sparse neighbours
	in $\bigcup_{i\in [\ell-1]}U_i^0$.
	\vspace{0.15cm}
	
	Since $X$ is strongly independent and
	$|V(H)\setminus(X\cup\textstyle\bigcup_iU_i^0)|
	\le\rho n+\frac{k\binom{\ell+1}{2}}{\eta}$,
	deleting the vertices outside $X\cup\bigcup_iU_i^0$ from $L_H(x)$  gives
	$|E(L_H(x)[\textstyle\bigcup_iU_i^0])|\ge t_{k-1}(n-s,\ell-1)-3\rho n^{k-1}$.
	Recall that $|E(K^0)|\le t_{k-1}(n-s,\ell-1)+\eta n^{k-1}$. We have
	$|E(K^0)\setminus E(L_H(x)[\textstyle\bigcup_iU_i^0])|\le3\rho n^{k-1}+\eta n^{k-1}$.	
	Let
	$U_{\rm bad}=\{u\in\textstyle\bigcup_iU_i^0: \{u,x\}\text{ is sparse}\}$.
	If $|U_{\rm bad}|>(4k\ell)^k\rho n$, then
	\[
	|E(K^0)\setminus E(L_H(x)[\textstyle\bigcup_iU_i^0])|>\frac1{k-1}|U_{\rm bad}|\left[\left(\frac{n-s}{2\ell}\right)^{k-2}-k\binom{\ell+1}{2}n^{k-3}\right]>3\rho n^{k-1}+\eta n^{k-1},
	\]
	a contradiction. This completes the proof of Claim 3.
	
	By Claim 3, there are at most $(4k\ell)^k\rho n|X|$ sparse pairs joining $X$ and
	$\bigcup_iU_i^0$. It follows that at most $(4k\ell)^{k/2}\sqrt\rho\,n$ vertices of
	$\bigcup_iU_i^0$ have more than $(4k\ell)^{k/2}\sqrt\rho\,|X|$ sparse neighbours in
	$X$. Delete these vertices, and let $U_i$ be the remaining subset of $U_i^0$. 
	Consequently, (v) holds. By \eqref{eq:54}, $|U_i|\geq |U_i^0|-(4k\ell)^{k/2}\sqrt\rho n\geq (1-\varepsilon)\frac{n-s}{\ell-1}$ for $i\in [\ell-1]$ and (i) follows. (ii) follows from Claim 1 and (iv) follows from Claim 3. Recall that each $u\in U_i^0$ has at most $\xi^{1/3}n$ sparse neighbours in $\bigcup_{j\ne i}U_j^0$ by the definition of $B_0$. Thus (vi) follows from the construction of $U_i$.
\end{proof}

Applying Lemmas \ref{lem:exp-structure} and \ref{lem:budget}, we derive the upper bound for $\ex_k\bigl(n,\{H_{\ell+1}^k\cup M_{s+1}^k\}\bigr)$.
\begin{lemma}\label{lem:augmented}
	Let $n,s,k,\ell$ be positive integers such that  $3\le k\le\ell$. There exist sufficiently large
	$s_0=s_0(k,\ell)$ and $n_0=n_0(k,\ell)$ such that the following holds. Let $H$ be  an
	$n$-vertex  $H_{\ell+1}^k$-free $k$-graph with $\nu(H)\le s$, where $s_0\le s<c_0n$ and $n\ge n_0$. Then $|E(H)|\le st_{k-1}(n-s,\ell-1)$.
\end{lemma}

\begin{proof}
	Let $\rho,\varepsilon,\xi,\eta$ be the constants provided by Lemma~\ref{lem:exp-structure}, and choose $s_0$ and $n_0$ sufficiently large so that Lemma~\ref{lem:exp-structure} and all inequalities in the proof of Lemma~\ref{lem:exp-structure} hold.
	Suppose, to the contrary, that $|E(H)|> st_{k-1}(n-s,\ell-1)$. By
	Lemma~\ref{lem:exp-structure}, there is a strongly independent set $X\subseteq V(H)$ with $|X|\ge s-\frac{k\binom{\ell+1}{2}}{\eta}$, together with pairwise disjoint sets $U_1',\ldots,U_{\ell-1}'\subseteq V(H)\setminus X$ satisfying the conclusions of Lemma~\ref{lem:exp-structure}.
	
	Let $t=|X|$,
	$V_2=\{x\in V(H):d_H(x)\le ks\binom{n-1}{k-2}\}$ and
	$V_1=V(H)\setminus(X\cup V_2)$.
	By Lemma~\ref{lem:exp-structure} and Lemma \ref{lem:exp-numerical}, $d_H(x)\ge t_{k-1}(n-s,\ell-1)-\eta n^{k-1}>ks\binom{n-1}{k-2}$ for every $x\in X$.
	Moreover, $X$ contains every vertex of degree at least $t_{k-1}(n-s,\ell-1)$ in $H$, so every vertex
	of $V_1$ has degree less than $t_{k-1}(n-s,\ell-1)$, as required in
	Lemma \ref{lem:budget}(ii). For each $(k-1)$-set $f\subseteq V_2$, define
	\begin{equation}
		m(f)=|\{x\in X:\{x\}\cup f\in E(H)\}|. \label{eq:55}
	\end{equation}
	Since $X$ is strongly independent, $\sum_{x\in X}d_{H[X\cup V_2]}(x)=\sum_{f\in\binom{V_2}{k-1}}m(f)$.
	
	To get the contradiction, by Lemma \ref{lem:budget}(ii), it suffices to prove
	\begin{equation}\label{eq:target-link-sum}
		\sum_{f\in\binom{V_2}{k-1}}m(f)
		\le t\,t_{k-1}(|V_2|,\ell-1).
	\end{equation}
	Since $d_H(x)>ks\binom{n-1}{k-2}$ for $x\in V_1$, $|X|+|V_1|\le s$ by Lemma~\ref{lem:budget}(i). Hence
	$|V_1|\le k\binom{\ell+1}{2}/\eta$. Let $U_i=U_i'\setminus V_1$ for $i\in[\ell-1]$. Then $|U_i|\ge(1-3\varepsilon)(n-s)/(\ell-1)$ for sufficiently large $n$.
	Let $R=V_2\setminus\bigcup_{i=1}^{\ell-1}U_i$. Then $|R|=|V_2|-\sum^{\ell-1}_{i=1}|U_i|\le n-|X|-|V_1|-(1-3\varepsilon)(n-s)\le4\varepsilon n$.
	If $R=\varnothing$, then $V_2=\bigcup_{i=1}^{\ell-1}U_i$. By Lemma~\ref{lem:exp-structure}(ii), $|e\cap U_i|\le1$ for each $e\in E(H)$ and each $i\in[\ell-1]$.
	Since $m(f)\le |X|= t$ for each $f\in \binom{V_2}{k-1}$, inequality \eqref{eq:target-link-sum} follows and the proof is done.

	So we assume that $R\ne\varnothing$ in the following proof.
	Among all partitions $\mathcal P=\{P_1,\ldots,P_{\ell-1}\}$ of $V_2$ such that $U_i\subseteq P_i$, choose one maximizing
	\begin{equation}
		\Phi(\mathcal P):=
		\sum_{f\in\binom{V_2}{k-1}}m(f)\cdot|\{i\in [\ell-1]:f\cap P_i\ne\varnothing\}|.\label{eq:56}
	\end{equation}
	Let $K$ be the complete $(\ell-1)$-partite $(k-1)$-graph on $P_1\cup \cdots\cup P_{\ell-1}$. Define
	\begin{align}
		\Bad(\mathcal P)&=
		\sum_{f\in\binom{V_2}{k-1}\setminus K}m(f), \label{eq:57}\\
		\Miss(\mathcal P)&=\sum_{f\in K}(t-m(f)). \label{eq:58}
	\end{align}
	If $\Bad(\mathcal P)\le\Miss(\mathcal P)$, then
	$$\sum_{f\in\binom{V_2}{k-1}}m(f)=\sum_{f\in\binom{V_2}{k-1}\setminus K}m(f)+\sum_{f\in K}m(f)\leq \sum_{f\in K}(t-m(f))+\sum_{f\in K}m(f)= t|E(K)|\leq t\,t_{k-1}(|V_2|,\ell-1)$$
	and the proof is done.
	
	Now we assume that
	$\Bad(\mathcal P)>\Miss(\mathcal P)$. We will eventually show that this cannot happen.
	
	\vspace{0.15cm}
	{\bf Claim.}	
	Let $x\in X$ and $f\in E(L_H(x)[V_2])\setminus E(K)$.
	Then there exists $i\in[\ell-1]$ satisfying $|f\cap P_i|\ge 2$, and a vertex $z\in f\cap P_i\cap R$ such that either
	\begin{equation}
		|\{y\in X: \{z,y\}\text{ is sparse}\}|>(1/2-2\varepsilon)|X|, \label{eq:59}
	\end{equation}
	or there exists $j\ne i$ such that
	\begin{equation}
		|\{u\in U_j: \{z,u\}\text{ is sparse}\}|>(1/2-2\varepsilon)|U_j|. \label{eq:510}
	\end{equation}

We now prove this claim.
Since $f\in\binom{V_2}{k-1}\setminus K$, there exists some $P_i$ containing two distinct vertices $a,b\in f$. Without loss of generality, we may assume that $i=1$.
By Lemma~\ref{lem:exp-structure}(ii), $|(\{x\}\cup f)\cap U_1|\leq 1$, so
$\{a,b\}\cap R\ne\varnothing$. Suppose, to the contrary, that no vertex of $\{a,b\}\cap R$ satisfies either \eqref{eq:59} or \eqref{eq:510}. For the case $\{a,b\}\subseteq R$, the number of their common dense neighbours in $X$ is at least $2\varepsilon|X|$. For the case $|\{a,b\}\cap R|=1$, the number of their common dense neighbours in $X$ is at least $(1/2+\varepsilon)|X|-(4k\ell)^{k/2}\sqrt\rho\,|X|\ge2\varepsilon|X|$ by  Lemma~\ref{lem:exp-structure}(v).
Hence in either case we can choose $y_1\in X\setminus \{x\}$ that is a  dense neighbour of both $a$ and $b$.
For $j\ne 1$, let $Z_j$ consist of those $u\in U_j\setminus f$ for which $\{y_1,u\}$, $\{a,u\}$, and $\{b,u\}$ are dense.
For the case $\{a,b\}\subseteq R$, $|Z_j|\ge 4\varepsilon|U_j|-(4k\ell)^k\rho n-k\ge3\varepsilon|U_j|$ by Lemma~\ref{lem:exp-structure}(iv). For the case $|\{a,b\}\cap R|=1$, $|Z_j|\ge (1/2+2\varepsilon)|U_j|-(4k\ell)^k\rho n-\xi^{1/3}n-k\ge3\varepsilon|U_j|$ by Lemma~\ref{lem:exp-structure}(iv) and (vi). Thus in either case $|Z_j|\geq 3\varepsilon|U_j|$. For the case $\ell=3$, let $z_2\in Z_2$. For the case $\ell>3$,
let $\mathcal S$ be the family of sparse pairs joining distinct $Z_{j_1}$ and $Z_{j_2}$, where $j_1,j_2\in [\ell-1]\setminus \{1\}$. By Lemma~\ref{lem:exp-structure}(vi), $$|\mathcal S|\le \frac12\sum_{j\ne 1}\sum_{u\in Z_j}\xi^{1/3}n\le \frac12\xi^{1/3}n^2<\left(\frac{3\varepsilon n}{2(\ell-1)}\right)^2\le(\min_{j\ne 1}|Z_j|)^2.$$
By Lemma~\ref{lem:selection}, there exist $z_j\in Z_j$ for each $j\ne 1$ such that every pair from $\bigcup_{j\neq 1}\{z_j\}$ is dense in $H$.
In either case, let $C=\{y_1,a,b,z_2,\ldots,z_{\ell-1}\}$.  Then every pair of $C$ is dense  in $H$ except possibly $\{a,b\}$, while
$(\{x\}\cup f)\cap C=\{a,b\}$. Lemma~\ref{obs:extension} yields a
copy of $H_{\ell+1}^k$  in $H$, a contradiction. This  completes the proof of the Claim.

Let $n^*=\min_i|U_i|$. Then $n^*\ge n/(2(\ell-1))$. Let $H_b$ be the
$k$-graph with edge  set $\{\{x\}\cup f: x\in X \mbox{ and } f\in E(L_H(x)[V_2])\setminus K\}$. Then
$|E(H_b)|=\Bad(\mathcal P)$.
Let
$R_0\subseteq R$ consist of the vertices satisfying \eqref{eq:59} or
satisfying \eqref{eq:510}.
Let $$\mathcal B=\bigl\{\{u,v\}: \exists\, i\in[\ell-1]\text{ and }e\in E(H_b)
\text{ such that }\{u,v\}\in\binom{P_i\cap e}{2}
\text{ and }\{u,v\}\cap R_0\ne\varnothing\bigr\}.$$
By the Claim, every edge of
$H_b$ contains a member of $\mathcal B$.

Since $t=|X|$, conditions \eqref{eq:59} and \eqref{eq:510} imply that there are at least $(1/2-2\varepsilon)t$ or $(1/2-2\varepsilon)n^*$ sparse neighbours for every vertex of $R_0$.
Thus the definition of sparse pairs gives
\[
\begin{aligned}
	\Miss(\mathcal P)
	&\ge\frac{|R_0|}{k-1}\cdot\min\Biggl\{
	(1/2-2\varepsilon)t
	\left(\left(\frac{n-t}{2\ell}\right)^{k-2}-\tau\right),
	(1/2-2\varepsilon)n^*
	\left(t\left(\frac{n-t}{2\ell}\right)^{k-3}-\tau\right)
	\Biggr\}\\
	&\ge |R_0|\frac{t}{4k^2}
	\left(\frac{n-t}{2\ell}\right)^{k-2},
\end{aligned}
\]
where $\tau=k\binom{\ell+1}{2}n^{k-3}$.
Since $|E(H_b)|=\Bad(\mathcal P)>\Miss(\mathcal P)$, we have
\[
\sum_{\{u,v\}\in\mathcal B}d_{H_b}(u,v)=\sum_{e\in E(H_b)}\left|\mathcal B\cap \binom{e}{2}\right|
\ge |R_0|\frac{t}{4k^2}
\left(\frac{n-t}{2\ell}\right)^{k-2}.
\]
For each $v\in R_0$, let $\mathcal B_v=\{\{v,w\}:\{v,w\}\in\mathcal B\}$. Since
$\mathcal B\subseteq\bigcup_{v\in R_0}\mathcal B_v$, $\sum_{v\in R_0}\sum_{\{v,w\}\in\mathcal B_v}d_{H_b}(v,w)\geq \sum_{\{u,v\}\in\mathcal B}d_{H_b}(u,v)$. By averaging, there exists
$z\in R_0$, say $z\in P_1$, such that
\begin{equation}
	\sum_{\{z,w\}\in\mathcal B_z}d_{H_b}(z,w)
	\ge\frac{t}{4k^2}
	\left(\frac{n-t}{2\ell}\right)^{k-2}. \label{eq:511}
\end{equation}
It follows that
\begin{equation}
	d_{H_b}(z)\ge\frac{t}{4k^3}
	\left(\frac{n-t}{2\ell}\right)^{k-2}. \label{eq:512}
\end{equation}
We claim that $z$ is a dense neighbour of at least $8\varepsilon|X|$ vertices of $X$. Otherwise,
$d_{H_b}(z)\le8\varepsilon t n^{k-2}+tk\binom{\ell+1}{2}n^{k-3}<16\varepsilon t n^{k-2}$,
contradicting \eqref{eq:512} when $\varepsilon$ is sufficiently small.

Let $W_z:=\{w:\{z,w\}\in\mathcal B_z\}$. Thus $W_z\subseteq P_1$.  
Since $d_{H_b}(z,w)\le tn^{k-3}$ for each $\{z,w\}\in\mathcal B_z$, $|W_z|tn^{k-3}\geq \sum_{\{z,w\}\in\mathcal B_z}d_{H_b}(z,w)$. It follows that $|W_z|\ge\frac{n}{4k^2(4\ell)^{k-2}}$ by inequality \eqref{eq:511}.
Since $P_1\setminus U_1\subseteq R$ and $|R|\le4\varepsilon n$,
for sufficiently small $\varepsilon$, we have $|W_z\cap U_1|\ge\frac{n}{8k^2(4\ell)^{k-2}}>0$.

We next claim that, for some $j\ne1$, $z$ is a dense neighbour of fewer than
$8\varepsilon|U_j|$ vertices of $U_j$. Otherwise, choose
$w\in W_z\cap U_1$ and an edge $e\in E(H_b)$ containing $z,w$. Since $z$
is a dense neighbour of at least $8\varepsilon|X|$ vertices of $X$ and, by Lemma \ref{lem:exp-structure}(v), $w$ has at most
$(4k\ell)^{k/2}\sqrt\rho\,|X|$ sparse neighbours in $X$, the number of  common dense neighbours
of $z$ and $w$
in $X$ is at least $(8\varepsilon-(4k\ell)^{k/2}\sqrt\rho)|X|>k$. Hence there is a
$y\in X\setminus e$ that is a dense neighbour of both $z$ and $w$. For $j\ne1$, let
$Z_j=\{u\in U_j\setminus e:\{z,u\},\{w,u\}\text{ and }\{y,u\}\text{ are dense in $H$}\}$.
By Lemma~\ref{lem:exp-structure}(iv) and (vi), and  the fact that
$|U_j|\ge n^*\ge n/(2(\ell-1))$, we have
$|Z_j|\ge 8\varepsilon|U_j|-(4k\ell)^k\rho n-\xi^{1/3}n-k\ge 8\varepsilon|U_j|-4\varepsilon|U_j|=4\varepsilon|U_j|$ for sufficiently large $n$. For the case $\ell=3$, let $u_2\in Z_2$. For the case $\ell>3$, let $\mathcal S'$ be the set of sparse pairs joining distinct $Z_{j_1}$ and $Z_{j_2}$ for $j_1,j_2\in [\ell-1]\setminus \{1\}$. By Lemma~\ref{lem:exp-structure}(vi), $|\mathcal S'|\le \frac12\sum_{j\ne 1}\sum_{u\in Z_j}\xi^{1/3}n\le \frac12\xi^{1/3}n^2$. By Lemma~\ref{lem:selection}, there exist $u_j\in Z_j$, $j\ne1$, so that every pair among  these $u_j$'s is dense in
$H$. In either case, let $C=\{y,z,w,u_2,\ldots,u_{\ell-1}\}$. Thus all pairs of $C$ are dense in $H$ except possibly $\{z,w\}$, while $e\cap C=\{z,w\}$ . Lemma~\ref{obs:extension} yields a
copy of $H_{\ell+1}^k$ in $H$, a contradiction.

Fix such an index $j$, move $z$ from $P_1$ to $P_j$, and denote the
resulting partition by $\mathcal P'$. Every edge whose
contribution to $\Phi$ decreases under this move contains $z$ and intersects
$P_j$. By the choice of $j$,  the number of such edges is at most
\[
8\varepsilon|U_j|tn^{k-3}+|P_j\setminus U_j|tn^{k-3}+|U_j|k\binom{\ell+1}{2}n^{k-3}
<24\varepsilon tn^{k-2}.
\]
Every edge counted on the left-hand side of \eqref{eq:511} contains
$z$ and another vertex of $P_1$, and each such edge is counted at
most $k-2$ times. After discarding the edges intersecting $P_j$,
the contribution of each remaining edge to $\Phi$ increases by $1$
under this move. By \eqref{eq:511} and the choice of $\varepsilon$,
\[
\Phi(\mathcal P')-\Phi(\mathcal P)
\ge\frac{1}{k-2}
\sum_{\{z,w\}\in\mathcal B_z}d_{H_b}(z,w)
-48\varepsilon tn^{k-2}>0.
\]
This contradicts the
maximality of $\mathcal P$. The proof is complete.
\end{proof}

\begin{proof}[Proof of Theorem~\ref{thm:small-expansion}]
The lower bound follows from the $k$-graph  $H(n,\ell,s,k)$, while the upper bound follows from Lemma~\ref{lem:augmented}.
\end{proof}

\section{Concluding remarks}

Theorems~\ref{thm:intro-hypergraph}, \ref{thm:intro-boundary}, \ref{thm:small-s}, and \ref{thm:small-expansion} exhibit two
different regimes for $s$. The extremal constructions are different for the case that $s\leq c_0n$ and the case that $n/k-\beta n \leq s\leq n/k$ when $\ell>k$.
The following construction provides an additional candidate for $c_0n\leq s\leq n/k-\beta n$.

\begin{definition}\label{def:F-construction}
\rev{Let $\ell\ge k\ge3$ and $n\ge\ell+s$. Let $U$ and $V$ be disjoint sets
	with $|U\cup V|=n$ and $|U|=s$. Partition $U$ into $U_1,U_2$ and $V$ into
	$V_1,\ldots,V_{\ell-2}$, with the parts in each partition as equal in size
	as possible. Define $F(n,\ell,s,k)$ on $U\cup V$ by
	\[
	E(F(n,\ell,s,k))=
	\left\{e\in\binom{U\cup V}{k}:
	e\cap U\ne\varnothing,\ |e\cap U_i|\le1\ (i\in[2]),\
	|e\cap V_j|\le1\ (j\in[\ell-2])\right\}.
	\]}
	\end{definition}
	
	The $k$-graph $F(n,\ell,s,k)$ is $\ell$-partite, and hence
	$\cK_{\ell+1}^k$-free. Moreover, every edge meets $U$, and so
	$\nu(F(n,\ell,s,k))\le s$.
	A straightforward counting argument determines when
	$F(n,\ell,s,k)$ has at least as many edges as $H(n,\ell,s,k)$.
	Indeed,
	\begin{align*}
|E(F(n,\ell,s,k))|
&=s\binom{\ell-2}{k-1}
\left(\frac{n-s}{\ell-2}\right)^{k-1}
+\frac{s^2}{4}\binom{\ell-2}{k-2}
\left(\frac{n-s}{\ell-2}\right)^{k-2}
+O(n^{k-1}),\\
|E(H(n,\ell,s,k))|
&=s\binom{\ell-1}{k-1}
\left(\frac{n-s}{\ell-1}\right)^{k-1}
+O(n^{k-1}).
\end{align*}
After cancelling the common factor $s(n-s)^{k-2}$ from the two
leading terms, their comparison is equivalent to
$s/n\ge c_{\ell,k}$, where
\[
c_{\ell,k}=
\frac{4\bigl((\ell-2)^{k-1}-(\ell-k)(\ell-1)^{k-2}\bigr)}
{4\bigl((\ell-2)^{k-1}-(\ell-k)(\ell-1)^{k-2}\bigr)
+(k-1)(\ell-2)(\ell-1)^{k-2}}.
\]
Consequently, for every fixed $\eta>0$ and all sufficiently large $n$,
$|E(F(n,\ell,s,k))|\ge |E(H(n,\ell,s,k))|$
whenever $(c_{\ell,k}+\eta)n\le s\le\frac{n-k}{k}$.

\section*{Acknowledgements}

This work is supported by the National Natural Science Foundation of China (Nos. 12501499 and 12571381) and Natural Science Foundation of Henan Province (No. 262300421874)

\appendix

\section{Proofs of Lemma~\ref{lem:numerical} and Lemma~\ref{lem:exp-numerical}}\label{app:numerical-proofs}


\begin{proof}[Proof of Lemma~\ref{lem:numerical}]
	Since $s\ge1$, we have
	$\lfloor(n-s)/(\ell-1)\rfloor
	\ge(n-(\ell-1)s)/(\ell-1)$.
	Moreover, $s<c_0n$ implies
	$(k-1)(\ell-1)s/n<1/2$.
	By Bernoulli's inequality and
	$\binom{\ell-1}{k-1}/(\ell-1)^{k-1}
	\ge1/(k-1)^{k-1}$, we obtain
	\[
	\begin{aligned}
		t_{k-1}(n-s,\ell-1)\ge
		\frac{\binom{\ell-1}{k-1}}{(\ell-1)^{k-1}}
		\left(1-\frac{(\ell-1)s}{n}\right)^{k-1}n^{k-1}\ge
		\frac12
		\frac{\binom{\ell-1}{k-1}}{(\ell-1)^{k-1}}n^{k-1}
		\ge\frac{n^{k-1}}{2(k-1)^{k-1}}.
	\end{aligned}
	\]
By	$\binom{n-1}{k-2}\le n^{k-2}/(k-2)!$, $k^2+1\le3(k-1)^2$ and $\ell> k$, we can derive that
	\[
	\begin{aligned}
		(k^2+1)s\binom{n-1}{k-2}<\frac{3(k-1)^2}
		{8(k-1)^{k-2}(\ell-1)^3(k-2)!}\,n^{k-1}\le\frac{3n^{k-1}}{8(k-1)^{k-1}}.
	\end{aligned}
	\]
	Consequently,
	\begin{equation}\label{eq:numerical-margin}
		t_{k-1}(n-s,\ell-1)
		-(k^2+1)s\binom{n-1}{k-2}
		>
		\frac{n^{k-1}}{8(k-1)^{k-1}},
	\end{equation}
	which proves $	t_{k-1}(n-s,\ell-1)>(k^2+1)s\binom{n-1}{k-2}$.
	
	We next prove
	\[
	\frac{\binom{\ell-1}{k-1}}{(\ell-1)^{k-1}}
	-\frac{\binom{\ell-2}{k-1}}{(\ell-2)^{k-1}}
	\ge
	\frac{k-2}{2(k-1)^{k-2}(\ell-1)^2}.
	\]
	Define
	$f(u)=\frac{1}{(k-1)!}
	\prod_{i=1}^{k-2}(1-i/u)$ for $u\ge k-1$.
	For every $u\in[\ell-2,\ell-1]$, we have
	$f(u)\ge1/(k-1)^{k-1}$ and
	\[
	\frac{f'(u)}{f(u)}
	=\sum_{i=1}^{k-2}\frac{i}{u(u-i)}
	\ge
	\frac{(k-1)(k-2)}{2(\ell-1)^2}.
	\]
	Thus
	$f'(u)\ge(k-2)/(2(k-1)^{k-2}(\ell-1)^2)$,
	and the desired result follows from the mean value theorem.
	
	By Bernoulli's inequality,
	\[
	\begin{aligned}
		&t_{k-1}(n-s,\ell-1)
		-2s\binom{n-1}{k-2}
		-t_{k-1}(n-2s,\ell-2)\\
		\ge& n^{k-1}\Biggl[
		\frac{\binom{\ell-1}{k-1}}{(\ell-1)^{k-1}}
		-\frac{\binom{\ell-2}{k-1}}{(\ell-2)^{k-1}}
		-\frac{s}{n}\left(
		\frac{(k-1)\binom{\ell-1}{k-1}}{(\ell-1)^{k-2}}
		+\frac{2}{(k-2)!}
		\right)
		\Biggr].
	\end{aligned}
	\]
	Since
	$\binom{\ell-1}{k-1}/(\ell-1)^{k-1}
	\le1/(k-1)!$, 
	$\frac{(k-1)\binom{\ell-1}{k-1}}{(\ell-1)^{k-2}}
	+\frac{2}{(k-2)!}
	\le\frac{\ell+1}{(k-2)!}
	\le\frac{2(\ell-1)}{(k-2)!}$.
	Furthermore,
	$\frac{2(\ell-1)s}{n(k-2)!}
	<
	\frac{1}{4(k-1)^{k-2}(\ell-1)^2(k-2)!}
	\le
	\frac{k-2}{4(k-1)^{k-2}(\ell-1)^2}$.
	Thus
	\begin{equation}\label{eq:numerical-gap}
		\begin{aligned}
			t_{k-1}(n-s,\ell-1)
			-2s\binom{n-1}{k-2}
			-t_{k-1}(n-2s,\ell-2)>
			\frac{k-2}{4(k-1)^{k-2}(\ell-1)^2}\,n^{k-1}.
		\end{aligned}
	\end{equation}
	
	By the assumption $s<c_0n$,
	\[
	\begin{aligned}
		ks\binom{n-1}{k-2}<
		\frac{k}{8(k-1)^{k-2}(\ell-1)^3(k-2)!}\,n^{k-1}<
		\frac{k-2}{4(k-1)^{k-2}(\ell-1)^2}\,n^{k-1},
	\end{aligned}
	\]
	where the last inequality follows from
	$k<2(\ell-1)(k-2)(k-2)!$.
	Combining with inequality \eqref{eq:numerical-gap}, we conclude that $t_{k-1}(n-s,\ell-1)-(k+2)s\binom{n-1}{k-2}>t_{k-1}(n-2s,\ell-2)$.
\end{proof}


\begin{proof}[Proof of Lemma~\ref{lem:exp-numerical}]
	Since $\eta<1/(8(k-1)^{k-1})$,
	inequality $t_{k-1}(n-s,\ell-1)-\eta n^{k-1}>(k^2+1)s\binom{n-1}{k-2}$ follows directly from inequality
	\eqref{eq:numerical-margin}.
	
	Note that $n-s-\rho n\ge n/2$.
	Using similar arguments as in the proof of Lemma \ref{lem:numerical}, we obtain
	\[
	\begin{aligned}
		&t_{k-1}(n-s,\ell-1)
		-t_{k-1}(\lfloor n-s-\rho n\rfloor,\ell-1)\\
		&\quad=
		\frac{\binom{\ell-1}{k-1}}{(\ell-1)^{k-1}}
		\bigl((n-s)^{k-1}-(n-s-\rho n)^{k-1}\bigr)
		+O(n^{k-2})\\
		&\quad\ge
		\frac{\rho}{2^{k-1}(k-1)^{k-2}}\,n^{k-1}
		>\eta n^{k-1}
	\end{aligned}
	\]
	for sufficiently large $n$.
	This proves $t_{k-1}(n-s,\ell-1)-\eta n^{k-1}>t_{k-1}(n-s-\rho n,\ell-1)$.
	
	Finally, by inequality \eqref{eq:numerical-gap}, we have
	\[
	\begin{aligned}
		&t_{k-1}(n-s,\ell-1)-\eta n^{k-1}
		-(2s+2\rho n+k)\binom{n-1}{k-2}
		-t_{k-1}(\lfloor n-2s-2\rho n-k\rfloor,\ell-2)\\
		>&
		n^{k-1}\left(
		\frac{k-2}{4(k-1)^{k-2}(\ell-1)^2}
		-\eta-\frac{2\rho+k/n}{(k-2)!}
		\right).
	\end{aligned}
	\]
	By the assumptions on $\rho$ and $\eta$, 
	$\eta+\frac{2\rho}{(k-2)!}<3\rho<\frac{3(k-2)}{64(k-1)^{k-2}(\ell-1)^2}$.
	Hence $t_{k-1}(n-s,\ell-1)-\eta n^{k-1}
	-(2s+2\rho n+k)\binom{n-1}{k-2}>t_{k-1}(\lfloor n-2s-2\rho n-k\rfloor,\ell-2)$.
\end{proof}

\end{document}